\documentclass[a4paper,11pt]{article}
\usepackage[pdftex]{graphicx}
\usepackage{graphicx,a4wide,listings}
\usepackage[intlimits]{amsmath}
\usepackage{amsxtra, amssymb,fancyhdr, amsthm,latexsym}
\usepackage[english]{babel}
\usepackage[ansinew]{inputenc}
\usepackage[bookmarks]{}
\usepackage{enumerate}
\numberwithin{equation}{section}
\numberwithin{figure}{section}
\usepackage[colorlinks=true]{hyperref}
\hypersetup{urlcolor=blue, citecolor=red}
\usepackage{enumitem}
\usepackage{tikz}
\usepackage{caption}
\usepackage{subfigure}

\newtheorem{Theorem}{Theorem}[section]

\newtheorem{Lemma}[Theorem]{Lemma}
\newtheorem{Proposition}[Theorem]{Proposition}

\theoremstyle{remark}
\newtheorem{rmk}{Remark}[section]

\theoremstyle{definition}
\newtheorem{Definition}{Definition}[section]

\newcommand{\nn}{\nonumber}
\newcommand{\R}{{\mathbb R}}

\newcommand{\N}{{\mathbb N}}
\newcommand{\Rd}{\R^d}
\newcommand{\Rk}{\R^k}
\newcommand{\Grad}{\nabla_{\!x}}
\newcommand{\del}{\partial}
\newcommand{\dx}{ \, {\rm d} x}

\newcommand{\dt}{ \, {\rm d} t}
\newcommand{\dv}{ \, {\rm d} v}

\newcommand{\CalP}{{\mathcal{P}}}

\newcommand{\Gradv}{\nabla_{\!v}}
\newcommand{\Eps}{\epsilon}

\newcommand{\Epsf}{f_\Eps}
\newcommand{\Epsfn}{f_{\Eps}^{(n)}}

\newcommand{\Epsrho}{\rho_\Eps}
\newcommand{\Epsrhon}{\rho_{\Eps}^{(n)}}

\newcommand{\abs}[1]{\left\lvert#1\right\rvert}
\newcommand{\norm}[1]{\left\lVert#1\right\rVert}

\newcommand{\vint}[1]{\langle #1 \rangle}
\newcommand{\vpran}[1]{\left( #1 \right)}

\begin{document}

\title{First-order aggregation models and zero inertia limits}
\author{R. C. Fetecau  \thanks{Department of Mathematics, Simon Fraser University, 8888 University Dr., Burnaby, BC V5A 1S6, Canada. Email: van@math.sfu.ca}
\and W. Sun \thanks{Department of Mathematics, Simon Fraser University, 8888 University Dr., Burnaby, BC V5A 1S6, Canada. Email: weirans@math.sfu.ca}}
\maketitle

\begin{abstract}
We consider a first-order aggregation model in both discrete and continuum formulations and show rigorously how it can be obtained as zero inertia limits of second-order models. In the continuum case the procedure consists in a macroscopic limit, enabling the passage from a kinetic model for aggregation to an evolution equation for the macroscopic density. We work within the general space of measure solutions and use mass transportation ideas and the characteristic method as essential tools in the analysis.

\vspace{0.5cm}
\end{abstract}


\section{Introduction}
\label{sect:intro}

The focus of the present paper is a certain mathematical model for emerging self-collective behaviour in biological (and other) aggregations.  There has been a surge of activity in this area of research during the past decade, and in fact the goals have extended well beyond biology. For biological applications, the primary motivation has been to understand and model the mechanisms behind the formation of the various spectacular groups observed in nature (fish schools, bird flocks, insect swarms) \cite{Camazine_etal}. In terms of expansion of this research into collateral areas, we mention studies on robotics and space missions \cite{JiEgerstedt2007}, opinion formation \cite{MotschTadmor2014},  traffic and pedestrian flow \cite{Helbing2001} and social networks \cite{Jackson2010}. 

Aggregation models can be classified in two main classes: i) individual/ particle-based, where the movements of all individuals in the group are being tracked, and ii) partial differential equations (PDE) models, formulated as evolution equations for the population density field. We refer to \cite{CarrilloVecil2010} for a recent review of models for aggregation behaviour, where the various microscopic/ macroscopic descriptions of collective motion are discussed and connected. In the present work we deal with a model that has both a discrete/ODE and a continuum/PDE formulation.

The continuum aggregation model considered in this article is given by the following evolution equation for the population density $\rho(t,x)$ in $\Rd$:
\begin{subequations}
\label{eq:aggre}
 \begin{align}
\rho_{t} &+\nabla\cdot(\rho u)=0, \label{eq:aggre1}\\
&u=-\nabla K\ast\rho, \label{eq:aggre2}
 \end{align}
\end{subequations}
where $K$ represents an interaction potential and $\ast$ denotes convolution. The potential $K$ typically incorporates social interactions such as short-range repulsion and long-range attraction. We consider $K$ to be radial, meaning that the inter-individual interactions are assumed to be isotropic.

Equation \eqref{eq:aggre} appears in various contexts related to mathematical models for biological aggregations; we refer
to \cite{M&K, TBL} and references therein for an extensive background and
review of the  literature on this topic. It also arises in a number of
other applications such as material science and granular media
\cite{Toscani2000}, self-assembly of nanoparticles \cite{HoPu2005} and molecular dynamics simulations of matter \cite{Haile1992}. The model has become widely popular and there has been intensive research on it during recent years. 

The particular appeal of model \eqref{eq:aggre} has lain in part in its simple form, which allowed rapid progress in terms of  both numerics and analysis. Numerical simulations demonstrated a wide variety of self-collective or ``swarm" behaviours captured by model \eqref{eq:aggre}, resulting in aggregations on disks, annuli, rings, soccer balls, etc \cite{KoSuUmBe2011, Brecht_etal2011, BrechtUminsky2012}.  Analysis-oriented studies addressed the well-posedness of the initial-value problem for \eqref{eq:aggre} \cite{BodnarVelasquez2, Burger:DiFrancesco, Laurent2007, BertozziLaurent,Figalli_etal2011,BeLaRo2011}, as well as the long time behaviour of its solutions \cite{Burger:DiFrancesco, LeToBe2009, FeRa10, BertozziCarilloLaurent, FeHuKo11,FeHu13}. Also, there has been increasing interest lately on the analysis of \eqref{eq:aggre} by variational methods \cite{BaCaLaRa2013, Balague_etalARMA, ChFeTo2014}.


Equation \eqref{eq:aggre} is frequently regarded  as the continuum approximation, when the number of particles increases to infinity, of the following individual-based model.  Consider $N$ particles in $\mathbb{R}^d$ whose positions $x_i$ ($i=1,\dots,N$) evolve according to the ODE system
\begin{subequations}
\label{eqn:fo-nobz}
\begin{align}
\frac{dx_i}{dt} &= v_i, \\
v_i &= -\frac{1}{N} \displaystyle\sum_{j\neq i}\nabla_{x_i}K(x_i-x_j), \label{eqn:vi-nobz}
\end{align}
\end{subequations}
where $K$ denotes the same interaction potential as in \eqref{eq:aggre}.

Model \eqref{eqn:fo-nobz} was justified and formally derived in \cite{BodnarVelasquez1}, starting from the following second-order model in Newton's law form ($i=1,\dots,N$):
\begin{equation}
\label{eqn:Nform}
\Eps \frac{d^2 x_i}{dt^2} + \frac{dx_i}{dt} = F_i, \qquad \text{ with } \quad F_i=-\frac{1}{N} \displaystyle\sum_{j\neq i}\nabla_{x_i}K(x_i-x_j), 
\end{equation}
and $\Eps>0$ small. From a biological point of view, \eqref{eqn:Nform} considers some small inertia/response time of individuals. By neglecting the $\Eps$-term in \eqref{eqn:Nform}, one can formally derive model \eqref{eqn:fo-nobz}. However, as noted in  \cite{BodnarVelasquez1}, making $\Eps=0$ translates to instantaneous changes in velocities, assumption which, quote, ``is probably too restrictive in many cases". 


In view of \eqref{eqn:fo-nobz}, one can write \eqref{eqn:Nform} more conveniently as
\begin{subequations}
\label{eqn:so-nobz}
\begin{align}
    \frac{dx_i}{dt} &= v_i, \\
    \Eps\dfrac{dv_i}{dt} &= -v_i -\dfrac1N\displaystyle\sum_{j\neq i}\nabla_{x_i}K(x_i-x_j).
      \end{align}
\end{subequations}
We point out that despite being at the origin of the extensively-studied models \eqref{eqn:fo-nobz} and \eqref{eq:aggre}, the second-order model \eqref{eqn:Nform} (or \eqref{eqn:so-nobz}) and in particular, the $\Eps \to 0$ limit of its solutions, have been overlooked completely. 
As briefly demonstrated in Section \ref{sect:odelimit}, a rigorous passage from model  from \eqref{eqn:so-nobz} to \eqref{eqn:fo-nobz} can be obtained in the $\Eps \to 0$ limit by using a classical theorem due to Tikhonov  \cite{Tikhonov1952}.

The main focus of the present work is the analogous $\Eps \to 0$ limit at the PDE level. Specifically, we investigate a zero inertia limit that yields the continuum model \eqref{eq:aggre}. Using techniques reviewed in \cite{CarrilloVecil2010}, one can formally take the limit $N\to \infty$ and associate to \eqref{eqn:so-nobz} the following kinetic equation for  the density $f(t,x,v)$ of individuals at position $x \in \Rd$ with velocity $v\in \Rd$:
\begin{equation}
\label{eqn:kinetic}
f_t + v \cdot \nabla_x f = \frac{1}{\epsilon} \Gradv \cdot (v f) + \frac{1}{\epsilon} \Gradv \cdot \left ( (\nabla_x K \ast \rho) f \right),
\end{equation}
where 
\begin{equation}
\label{eqn:rhof}
\rho(t,x) = \int_{\Rd} f(t,x,v) dv.
\end{equation}

We consider measure-valued solutions of the kinetic model \eqref{eqn:kinetic}, in the framework of the well-posedness theory developed in \cite{CanizoCarrilloRosado2011}, and study their macroscopic limit $\Eps\to 0$. Passage from kinetic to macroscopic equations has been extensively studied in the hydrodynamic limits of the nonlinear Botlzmann equations for both classical and renormalized solutions. It is beyond the scope of this introduction to give a detailed account of the work that has been done in this vast and well-established research area, we simply refer here to a recent review paper \cite{SR2014} and the references therein.

Our main result (see Theorems \ref{thm:conv-sol} and \ref{thm:rho-trans}) is that solutions $\Epsf(t,x,v)$ to \eqref{eqn:kinetic} converge weak-$^\ast$ as measures to $f(t,x,v) = \rho(t,x) \delta(v-u(t,x))$ as $\epsilon \to 0$, where $u$ is given in terms of $\rho$ by \eqref{eq:aggre2} and $\rho$ satisfies the continuum equation \eqref{eq:aggre1}. A recurring theme in the paper is the method of characteristics. In particular, we demonstrate how characteristic paths of \eqref{eqn:kinetic} limit as $\Eps \to 0$ to characteristic paths of  \eqref{eq:aggre} (see Theorem \ref{thm:conv-traj}). The subtlety lies in the fact that the limit is degenerate, as characteristics of a second-order system collapse to first-order characteristics.

A key motivation for the zero inertia limits investigated in this article is the following. The recent work \cite{EvFeRy2014} of one of the authors showed that the second-order model \eqref{eqn:Nform} is absolutely essential provided one wants to include anisotropy  in model \eqref{eqn:fo-nobz}. Specifically, \cite{EvFeRy2014} considers anisotropic inter-individual interactions in model \eqref{eqn:fo-nobz} by replacing the explicit representation \eqref{eqn:vi-nobz} of the velocities by a {\em weighted} sum, with weights that depend on a restricted visual perception of the individuals. Hence, these weights depend on the velocity vectors themselves, and the anisotropic analogue of \eqref{eqn:vi-nobz} becomes an {\em implicit} equation to be solved for $v_i$. It is shown in \cite{EvFeRy2014} that solutions of such implicit equations are generally non-unique and additionally, encounter discontinuities through the time evolution. The relaxation system \eqref{eqn:so-nobz}, along with its small inertia/response time, is proposed in  \cite{EvFeRy2014} as a biologically meaningful mechanism to select unique solutions and physically {\em correct} velocity jumps. 

As for the ODE case, the present study sets the stage for generalizations of the continuum model \eqref{eq:aggre} to include anisotropic interactions. In such an extension, \eqref{eq:aggre2} would become an implicit equation for $u$ and issues such as non-uniqueness and loss of smoothness are again expected to arise.
We argue that understanding how to approximate first-order models such as \eqref{eq:aggre} and \eqref{eqn:fo-nobz} (and subsequently, their generalizations) in the $\Eps \to 0$ limit of second-order models\footnote{Strictly speaking, \eqref{eqn:kinetic} is a first-order PDE, but we refer to it as a second-order model as it is essentially based on Newton's law \eqref{eqn:Nform}. Furthermore, for $\Eps>0$ fixed, the monokinetic closure of \eqref{eqn:kinetic} yields a momentum equation for the velocity, also in the form of Newton's second law \cite{CarrilloVecil2010}.}, is entirely essential for making further progress in this area of research.

Finally, we point out that  we work in this paper with smooth potentials $K$ that satisfy  $\nabla K \in W^{1,\infty}(\Rd)$. This assumption on $K$ is just slightly more restrictive than the smoothness conditions on $K$ assumed in the well-posedness theory from \cite{CanizoCarrilloRosado2011}, which is essentially used in this work.  The results in this paper  would not apply for instance to pointy potentials, such as Morse potentials. However, as noted in \cite{CanizoCarrilloRosado2011}, from the point of view of applications, it makes little difference to distinguish between a pointy potential and its smooth regularization, as the two would give qualitatively similar aggregation behaviour. 

The summary of the paper is as follows. Section \ref{sect:odelimit} presents the $\Eps\to 0$ limit of the ODE model \eqref{eqn:so-nobz}. Section \ref{sect:kinetic} contains a brief formal derivation of the kinetic model \eqref{eqn:kinetic} from \eqref{eqn:so-nobz} and summarizes the results from \cite{CanizoCarrilloRosado2011} regarding the well-posedness of measure-valued solutions of \eqref{eqn:kinetic}.  In Section \ref{sect:estimates} we derive uniform in $\Eps$ estimates for solutions to \eqref{eqn:kinetic} needed for passing the limit $\Eps \to 0$. Section \ref{sect:conv} contains the major results of this paper, which is the convergence of solutions to \eqref{eqn:kinetic} as $\Eps \to 0$ and how solutions of \eqref{eq:aggre} are recovered in this limit.


\section{Convergence as $\Eps \to 0$ of the ODE model \eqref{eqn:so-nobz}}
\label{sect:odelimit}

The limit $\Eps \to 0$ of solutions to \eqref{eqn:so-nobz} can be carried out by a straightforward application of the general theory originally developed by Tikhonov  \cite{Tikhonov1952}. An excellent account of this theory can be found in \cite{Vasileva1963}. Since the application of Tikhonov's theorem does not appear in any of the works on aggregation models, we find it worthwhile to summarize the concepts, as well as to state the convergence result in the context of models  \eqref{eqn:so-nobz} and \eqref{eqn:fo-nobz}. In fact, this framework will be used again in Section \ref{sect:conv}, when we study the $\epsilon \to 0$ limit of the PDE model \eqref{eqn:kinetic}.

Following the setup in \cite{Vasileva1963}, consider the general system 
\begin{equation}
\label{eqn:shorthand}
\left\{
  \begin{array}{l}
    \dfrac{d\mathbf{x}}{dt} = \mathbf{v},\\\\
    \Eps\dfrac{d\mathbf{v}}{dt} = \mathcal{F}(\mathbf{x},\mathbf{v}),
  \end{array}
\right.
\end{equation}
where $\mathbf{x},\mathbf{v} \in \mathbb{R}^{Nd}$ and $\Eps>0$.

Note that indeed, system \eqref{eqn:so-nobz} can be written compactly as \eqref{eqn:shorthand} provided $\mathbf{x},\mathbf{v}$
 denote the concatenation of the space and velocity vectors:
\begin{equation}
\label{eqn:concatenation}
\mathbf{x} = (x_1,\dots,x_N), \qquad \mathbf{v} = (v_1,\dots, v_N),
\end{equation}
and
\begin{equation}
\label{eqn:calF}
\mathcal{F}(\mathbf{x},\mathbf{v})=\left(\mathcal{F}_1(\mathbf{x},v_1),\ldots,\mathcal{F}_N(\mathbf{x},v_N)\right),
\end{equation}
with 
\begin{equation}
\label{eqn:calFi}
\mathcal{F}_i(\mathbf{x},v_i) =  -v_i -\dfrac1N\displaystyle\sum_{j\neq i}\nabla_{x_i}K(|x_i-x_j|), \qquad i=1,\dots,N.
\end{equation}

Tikhonov's result focuses on {\em roots} $\mathbf{v}=\Gamma(\mathbf{x})$ of the equation
\begin{equation}
\label{eqn:calFroots}
\mathcal{F}(\mathbf{x},\mathbf{v})=0.
\end{equation}
In particular, the goal is to establish conditions on a root $\Gamma$ which guarantee that solutions of \eqref{eqn:shorthand} converge to solutions of the {\em degenerate} system associated to $\Gamma$, defined as
\begin{equation}
\label{eqn:shorthand:fo}
\left\{
  \begin{array}{l}
    \dfrac{d\mathbf{x}}{dt} = \mathbf{v},\\\\
    \mathbf{v} = \Gamma(\mathbf{x}).
  \end{array}
\right.
\end{equation}

Using this terminology (see \eqref{eqn:concatenation}-\eqref{eqn:calFi}), model \eqref{eqn:fo-nobz} is the degenerate system associated to the {\em unique} root $\mathbf{v} =\Gamma(\mathbf{x})$, where $\Gamma(\mathbf{x}) = (\gamma_1(\mathbf{x}), \dots, \gamma_N(\mathbf{x}))$ is given {\em explicitly} by
\begin{equation}
\label{eqn:root}
v_i = \gamma_i(\mathbf{x}) := -\dfrac1N\displaystyle\sum_{j\neq i}\nabla_{x_i}K(|x_i-x_j|).
\end{equation}

In general, take a closed and bounded set $D\subset\R^{Nd}$, and consider a root $\mathbf{v}=\Gamma(\mathbf{x})$, $\Gamma:D\to\R^{Nd}$. The root $\Gamma$ is called \textit{isolated} if there is a $\delta>0$ such that for all $\mathbf{x}\in D$, the only element in $B(\Gamma(\mathbf{x}),\delta)$ that satisfies $\mathcal{F}(\mathbf{x},\mathbf{v})=0$ is $\mathbf{v}=\Gamma(\mathbf{x})$.

For a fixed configuration $\mathbf{x}^\ast$, the system
\begin{equation}\label{eqn:adj}
\dfrac{d\mathbf{v}}{d\tau}=\mathcal{F}(\mathbf{x}^\ast,\mathbf{v}),
\end{equation}
is called the \textit{adjoined system of equations}. 

An isolated root $\Gamma$ is called \textit{positively stable} in $D$, if $\mathbf{v}^\ast=\Gamma(\mathbf{x}^\ast)$ is an asymptotically stable stationary point of \eqref{eqn:adj} as $\tau\to\infty$, for each $\mathbf{x}^\ast\in D$. The \textit{domain of influence} of an isolated positively stable root $\Gamma$ is the set of points $(\mathbf{x}^\ast,\tilde{\mathbf{v}})$ such that the solution of \eqref{eqn:adj} satisfying $\mathbf{v}|_{\tau=0}=\tilde{\mathbf{v}}$ tends to $\mathbf{v}^\ast = \Gamma(\mathbf{x}^\ast)$ as $\tau\to\infty$.

Tikhonov's theorem  \cite{Tikhonov1952} states the following:

\begin{Theorem}[Tikhonov \cite{Tikhonov1952, Vasileva1963}] \label{thm:Tikhonov}
Assume that $\Gamma$ is an isolated positively stable root of \eqref{eqn:calFroots} in some bounded closed domain $D$. Consider a point $(\mathbf{x}_0,\mathbf{v}_0)$ in the domain of influence of this root, and assume that the degenerate system  \eqref{eqn:shorthand:fo} has a solution
 $\mathbf{x}(t)$ initialized at $\mathbf{x}(t_0) = \mathbf{x}_0$, that lies in $D$ for all $t\in[t_0,T]$. Then, as $\Eps\to 0$, the solution 
$(\mathbf{x}_\Eps (t),\mathbf{v}_\Eps (t))$ 
of \eqref{eqn:shorthand} initialized at $(\mathbf{x}_0,\mathbf{v}_0)$, converges to $(\mathbf{x}(t),\mathbf{v}(t)):=(\mathbf{x}(t),\Gamma(\mathbf{x}(t)))$ in the following sense:

i) $\displaystyle\lim_{\Eps\to 0}\mathbf{v}_\Eps(t)=\mathbf{v}(t)$ \text{ for all } $t\in(t_0,T^*]$, \text {and }
\smallskip

ii) $\displaystyle\lim_{\Eps\to 0}\mathbf{x}_\Eps(t)=\mathbf{x}(t)$ \text{ for all } $t\in[t_0,T^*]$,

\smallskip
for some $T^*<T$.
\end{Theorem}

\begin{rmk}\label{remark:boundary layer}
Note that the convergence of $\mathbf{v}_\Eps(t)$ to $\mathbf{v}(t)$ holds for $t>t_0$ and normally does not occur at the initial time $t_0$, unless $\mathbf{v}_0=\Gamma(\mathbf{x}_0)$.
\end{rmk}

Applying Theorem \ref{thm:Tikhonov} to models \eqref{eqn:so-nobz} and \eqref{eqn:fo-nobz} is immediate. Given any spatial configuration $\mathbf{x}$, the root $\Gamma$ given by \eqref{eqn:root} is unique, hence isolated. Fix now an arbitrary spatial configuration $\mathbf{x^\ast}=(x_1^\ast,x_2^\ast,\dots,x_N^\ast)$ and inspect the adjoint system \eqref{eqn:adj} with $\cal F$ given by \eqref{eqn:calF}-\eqref{eqn:calFi}. It is clear that each component of the adjoint system,
\[
\frac{ d v_i}{d \tau} = {\cal F}_i(\mathbf{x}^\ast,v_i),
\]
has a globally attracting equilibrium $v_i^\ast = \gamma_i(\mathbf{x}^\ast)$. Consequently, $\mathbf{v}^\ast = (v_1^\ast,\dots,v_N^\ast)$ is positively stable and its domain of influence is $\{\mathbf{x}^\ast\} \times \mathbb{R}^{Nd}$.

\begin{Theorem}[Convergence of the ODE model]
Consider a point $(\mathbf{x}_0,\mathbf{v}_0) \in \mathbb{R}^{2Nd}$, and suppose the first-order model \eqref{eqn:fo-nobz} has a solution $\mathbf{x}(t) = (x_1(t),\dots,x_N(t)) \in \mathbb{R}^{Nd}$ initialized at $\mathbf{x}_0$ for $t \in [t_0,T]$. Then, as $\Eps \to0$, the solution $(\mathbf{x}_\Eps(t),\mathbf{v}_\Eps(t)) \in \mathbb{R}^{2Nd}$ of the second-order model \eqref{eqn:so-nobz} initialized at $(\mathbf{x}_0,\mathbf{v}_0)$, converges to $(\mathbf{x}(t),\mathbf{v}(t))$, with $\mathbf{v}(t)=(v_1(t),\dots,v_N(t))$ defined in terms of $\mathbf{x}(t)$ by \eqref{eqn:root}. 

Specifically, we have the convergence i) and ii) listed in Theorem \ref{thm:Tikhonov}, with the caveat that the convergence of $\mathbf{v}_\Eps(t)$ does not hold initially, unless $\mathbf{v}_0$ and $\mathbf{x}_0$ are related by \eqref{eqn:root}.

\end{Theorem}


\section{Kinetic model and its well-posedness}
\label{sect:kinetic}

\subsection{Formal derivation of the kinetic model}

The kinetic equation associated to the particle model \eqref{eqn:so-nobz} can be derived using the techniques reviewed in \cite{CarrilloVecil2010}.  We present here the derivation via the mean-field limit \cite{CarrilloVecil2010}. 

Consider the distribution density $f_N$ associated to the solution $(x_i(t),v_i(t))$ $(i=1,\dots,N)$ of  \eqref{eqn:so-nobz}, that is,
\[
f_N(t,x,v) = \frac{1}{N} \sum_{i=1}^N \delta(x-x_i(t)) \delta(v-v_i(t)).
\]

Take a test function $\varphi \in C^1_0(\mathbb{R}^{2d})$ and compute, using \eqref{eqn:so-nobz}:
\begin{align*}
\frac{d}{dt} \langle f_N(t),\varphi \rangle 
&=  \frac{1}{N} \sum_{i=1}^N \frac{d}{dt} \varphi(x_i(t),v_i(t)) \\
&=  \frac{1}{N} \sum_{i=1}^N \nabla_x \varphi(x_i(t),v_i(t)) \cdot v_i(t) \\
& \quad +  \frac{1}{N } \sum_{i=1}^N \nabla_v \varphi(x_i(t),v_i(t)) \cdot \frac{1}{\Eps} \Bigl ( -v_i -\dfrac1N\displaystyle\sum_{j\neq i}\nabla_{x_i}K(|x_i-x_j|) \Bigr).
\end{align*}

Denote by $\rho_N(t,x)$ the macroscopic density of $f_N$:
\[
\rho_N(t,x) := \int f_N(t,x,v) dv = \frac{1}{N} \sum_{i=1}^N \delta(x-x_i(t)).
\] 
Since 
\begin{align*}
\nabla K  \ast \rho_N(x) &= \int \nabla K(x-y) \rho_N(y) dy  \\
& = \frac{1}{N}  \sum_{j=1}^N \nabla K(x-x_j),
\end{align*}
we get
\[
\frac{d}{dt} \langle   f_N(t),\varphi \rangle = 
\langle f_N(t), \nabla_x \varphi \cdot v\rangle + \left \langle f_N(t), \nabla_v \varphi \cdot \frac{1}{\Eps} \left ( -v - \nabla K \ast \rho_N (x) \right) \right \rangle.
\]
Hence, after integration by parts in $x$ and $v$ we get
\[
\left \langle \frac{\partial f_N}{\partial t} + v \cdot \nabla_x f_N - \frac{1}{\Eps} \nabla_v \cdot \left (v f_N + (\nabla K \ast \rho_N) f_N \right), \varphi \right \rangle = 0. 
\]
In strong form, $f_N$ satisfies
\[
\frac{\partial f_N}{\partial t} + v \cdot \nabla_x f_N = \frac{1}{\Eps} \nabla_v \cdot \left( v f_N\right) + \frac{1}{\Eps} \nabla_v \cdot \left( (\nabla K \ast \rho_N) f_N \right).
\]
Assuming that $f_N$ converges (on a subsequence) to a density $f$, taking the limit $N\to \infty$, formally, in the equation above, yields the kinetic equation \eqref{eqn:kinetic}.


\subsection{Well-posedness for \eqref{eqn:kinetic} with $\epsilon>0$ fixed.} We discuss first the well-posedness theory of measure-valued solutions of \eqref{eqn:kinetic}, as developped in \cite{CanizoCarrilloRosado2011}. Since for later purposes (to send $\Eps \to 0$) we need to work with smooth solutions, we present briefly the existence theory for classical solutions as well. A key ingredient is the method of characteristics, which is eventually used in  Section \ref{sect:conv} to connect the PDE analysis with the ODE theory via Tikhonov's theorem.

\paragraph{Measure solutions.} In \cite{CanizoCarrilloRosado2011} the authors consider various kinetic models for aggregation and study the well-posedness of measure-valued solutions. The results there use the following measure space setup. Denote by $\mathcal{P}_1(\Rk)$ the space of probability measures on $\Rk$ that have finite first moment, i.e.,
\[
\mathcal{P}_1(\Rk) = \biggl \{ f \in \mathcal{P}(\Rk):  \int_{\Rk} |x| f(x) dx < \infty \biggr \}.
\]
We note that the convention adopted in \cite{CanizoCarrilloRosado2011}, which is also used throughout the present paper, is to write $\int \varphi(x) \mu(x) dx $ as the integral of $\varphi$ with respect to the measure $\mu$, regardless of whether $\mu$ is absolutely continuous with respect to the Lebesgue measure.

\smallskip
 \begin{rmk}
\label{rmk:W1}
Endowed with the $1$-Wasserstein distance $W_1$, the space  $\mathcal{P}_1(\Rk)$ is a complete metric space, and convergence in the $W_1$ metric relates to the usual weak-$^*$ convergence of measures. Specifically, for $\{f_n\}_{n\geq 1}$ and $f$ in $\mathcal{P}_1(\Rk)$,  the following are equivalent:

\smallskip
i) $f_n \stackrel{W_1}{\longrightarrow} f$, as $n\to \infty$ 
\smallskip

ii) $f_n \stackrel{w^*}{\longrightarrow}  f$ as measures as $n\to \infty$ and
$\sup_{n\geq 1} \int_{|x|>R} |x| f_n(x) dx \to 0, \quad \text{ as } R\to \infty$.
\end{rmk}

Results in \cite{CanizoCarrilloRosado2011} use $\mathcal{P}_1(\mathbb{R}^d \times \Rd)$ endowed with the $W_1$ distance as the measure space where solutions of the various kinetic models  are sought for. Model \eqref{eqn:kinetic}, with $\Eps>0$ fixed, is in fact a particular case of the general class of models studied considered there and the results from \cite{CanizoCarrilloRosado2011} can be applied directly. We summarize briefly the results from  \cite{CanizoCarrilloRosado2011}. 

Consider the characteristic equations associated to model \eqref{eqn:kinetic}:
\begin{equation}
\begin{aligned} \label{eq:characteristics}
     \frac{{\rm d} x}{\dt} & \, = v \,,
\\
     \frac{{\rm d} v}{\dt} & \, = -\frac{1}{\Eps} v - \frac{1}{\Eps} \Grad K \ast \rho,
\end{aligned}
\end{equation}
initialized at $(x,v)_{\mid_{t=0}} = (x_0, v_0)$. The main idea in  \cite{CanizoCarrilloRosado2011} is to define a measure solution to \eqref{eqn:kinetic} in a mass transportation sense, using the flow map defined by \eqref{eq:characteristics}.

Suppose $E(t,x)$ is a given continuous vector field on $[0,T]\times \Rd$ which is locally Lipschitz with respect to $x$ . Take the characteristic system associated to $E$:
\begin{equation}
\begin{aligned} \label{eq:characteristicsE}
     \frac{{\rm d} x}{\dt} & \, = v \,,
\\
     \frac{{\rm d} v}{\dt} & \, = -\frac{1}{\Eps} v - \frac{1}{\Eps} E(t,x),
\end{aligned}
\end{equation}
with initial data $(x,v)_{\mid_{t=0}} = (x_0, v_0)$. Then standard ODE theory guarantees existence and uniqueness of smooth trajectories $(x_\Eps,v_\Eps) \in C^1([0,T], \Rd \times \Rd)$ originating from $(x_0,v_0)$. Furthermore, one can define the flow map $\mathcal{T}^{t,\Eps}_{E}$ of  \eqref{eq:characteristicsE} by
\[
(x_0,v_0) \xrightarrow{\mathcal{T}^{t,\Eps}_{E}} (x,v), \quad \qquad (x,v) = (x_\Eps(t),v_\Eps(t)),
\]
where $(x_\Eps(t),v_\Eps(t))$ is the unique solution of \eqref{eq:characteristicsE} that starts at $(x_0,v_0)$. 

Now take a measure $f_0 \in \mathcal{P}_1(\Rd \times \Rd)$ and $T>0$, and consider the mass-transport (or push-forward) of $f_0$ by $\mathcal{T}^{t,\Eps}_{E}$. By definition, the push-forward $f_t = \mathcal{T}^{t,\Eps}_{E} \# f_0$ is a measure-valued function $f:[0,T] \to \mathcal{P}_1(\Rd \times \Rd)$, that satisfies
\begin{equation}
\label{eqn:mass-transport}
\int_{\R^{2d}} \zeta (x,v) f(t,x,v) dx dv = \int_{\R^{2d}} \zeta(\mathcal{T}^{t,\Eps}_{E}(X,V)) f_0(X,V) dX dV,
\end{equation}
for all $\zeta \in C_b(\R^{2d})$. 

Return to \eqref{eq:characteristics} and define the vector field $E[f]$ associated to a measure $f$ as 
\begin{equation}
\label{eqn:Ef}
E[f] = -\nabla K \ast \rho.
\end{equation}
Here, $\rho$ denotes the first marginal of $f$ defined by
\begin{equation}
\label{eqn:1stm}
    \int_{\R^d} \tilde\psi(x) \rho(t, x) \dx = \int_{\R^{2d}} \tilde\psi(x) f (t, x, v) \dx\dv,
\end{equation}
for all $\tilde\psi \in C_b(\R^d)$. Note that throughout the paper, by an abuse of notation, we also write $\rho$ as in \eqref{eqn:rhof}.

The following definition of a measure solution of \eqref{eqn:kinetic}  is adopted from \cite{CanizoCarrilloRosado2011}. 
\begin{Definition}
\label{defn:sol}
Take an initial measure  $f_0 \in \mathcal{P}_1(\Rd \times \Rd)$ and $T>0$. A function $f:[0,T] \to \mathcal{P}_1(\Rd \times \Rd)$ is a solution of the kinetic equation \eqref{eqn:kinetic} with initial condition $f_0$ if:
\begin{enumerate}
\item The field $E[f]$ defined by \eqref{eqn:Ef}  is locally Lipschitz with respect to $x$ and $E[f](t,x)< C(1+|x|)$, for all $t,x \in [0,T]\times \Rd$, for some $C>0$, and
\item $f_t = \mathcal{T}^{t,\Eps}_{E[f]} \# f_0$.
\end{enumerate}
\end{Definition}

The main result in \cite{CanizoCarrilloRosado2011} establishes the existence and uniqueness for measure solutions via a fixed point argument. We state and discuss the result below.
\begin{Theorem}[Measure-valued solutions \cite{CanizoCarrilloRosado2011}] 
\label{thm:exist-measure}
Assume the following properties on the potential $K$:
\[
\nabla K \text{ is locally Lipschitz, and } |\nabla K(x)| \leq C(1+|x|) \text{ for all } x\in \Rd,
\]
for some $C>0$. Consider $f_0 \in \mathcal{P}_1(\Rd \times \Rd)$ with compact support. 

Then there exists a unique solution $\Epsf \in C([0,\infty); \mathcal{P}_1(\Rd \times \Rd))$ of \eqref{eqn:kinetic}, in the sense of Definition \ref{defn:sol}, whose support grows at a controlled rate. Specifically, there exists an increasing function $R_\Eps(T)$ such that for all $T>0$,
\[
\operatorname{supp }\Epsf(t) \subset B_{R_\Eps(T)} \quad  \text{ for all } t \in [0,T].
\]
\end{Theorem}
The proof of this result in \cite{CanizoCarrilloRosado2011} relies on a fixed point argument. Briefly, the setup in \cite[Theorem 3.10]{CanizoCarrilloRosado2011} is the following.  Fix $T>0$ and consider the metric space $\mathcal{F}$ made of all $f \in C([0,T]; \mathcal{P}_1(\Rd \times \Rd))$ such that the support of $f_t$ is contained in a fixed ball $B_R$ for all $t \in [0,T]$. The distance in $\mathcal{F}$ is taken to be
\[
\mathcal{W}_1(f,g) = \sup_{t\in [0,T]} W_1(f_t,g_t),
\]
where $W_1$ denotes the $1$-Wasserstein distance. For $f\in \mathcal{F}$ fixed, define the map (see \eqref{eqn:Ef}):
\[
\mathcal{G}[f](t) :=  \mathcal{T}^{t,\Eps}_{E[f]} \# f_0.
\]
It is shown in \cite{CanizoCarrilloRosado2011} that this map is contractive and hence, is has a unique fixed point in $\mathcal{F}$. This fixed point is the desired solution to \eqref{eqn:kinetic}. 

We also note that by results in \cite{AGS2005,Figalli_etal2011}, $f_\Eps$ given by Theorem \ref{thm:exist-measure} is also a weak solution of \eqref{eqn:kinetic}, i.e., it satisfies
\begin{equation} \label{eq:f-Eps-weak}
\begin{aligned}
& \int_0^T \iint_{\R^{d} \times \R^d} \del_t \phi(t, x, v) \Epsf(t, x, v) \dx\dv \dt
    + \int_0^T \iint_{\R^d \times \R^d} \Grad \phi \cdot v \Epsf \dx\dv\dt
\\
& \qquad - \frac{1}{\Eps} 
         \int_0^T \iint_{\R^d \times \R^d} 
             \nabla_v \phi \cdot (v + \Grad K \ast \Epsrho) \Epsf \dx\dv\dt
    + \iint_{\R^d \times \R^d} \phi(0, x, v) f_0(x, v) \dx\dv=0,
\end{aligned}
\end{equation}
for any $\phi \in C^1_c([0, T); C^1_b(\R^d \times \R^d))$. 



\smallskip
\begin{rmk}Take a measure solution $f_\Eps$ of \eqref{eqn:kinetic} for $\Eps>0$ fixed. By definition, $f_\Eps$ is the mass transport of $f_0$ along trajectories $(x_\Eps(t),v_\Eps(t))$ that satisfy the characteristic equations (see \eqref{eq:characteristics})
\begin{equation}
\begin{aligned} \label{eq:characteristics-eps}
     \frac{{\rm d} x}{\dt} & \, = v \,,
\\
     \Eps \frac{{\rm d} v}{\dt} & \, = - v -  \Grad K \ast \rho_\Eps,
\end{aligned}
\end{equation}
$(x_\Eps(0),v_\Eps(0))= (x_0,v_0)$.  We are interested in this paper in the $\Eps \to0$ limit of $\Epsf$, $\Epsrho$, as well as in the limiting behaviour of the characteristic trajectories $x_\Eps(t)$, $v_\Eps(t)$. This will require the use of the ODE framework from Section  \ref{sect:odelimit} combined with PDE estimates on  \eqref{eqn:kinetic} itself. To this end, we need to work first with classical solutions of \eqref{eqn:kinetic} and derive uniform in $\Eps$ estimates (see Theorem \ref{thm:estimate} for instance). Below is a brief account on existence theory for classical solutions of  \eqref{eqn:kinetic}.
\end{rmk}


\paragraph{Smooth solutions.} The existence of smooth solutions to~\eqref{eqn:kinetic} for $\Eps>0$ can be inferred using the classical framework for Vlasov type equations \cite{G}. We state the theorem below and explain the steps of its proof. The full details of the proof can be found in \cite[Chapter 4]{G} for Vlasov-Maxwell equation. 

\begin{Theorem}[Existence of smooth solutions]
\label{thm:exist-smooth}
Suppose $f_0 \in C^2(\Rd \times \Rd) \cap L^1(\Rd \times \Rd)$ and $\Grad K \in W^{1, \infty}(\Rd)$. Let $T > 0$ be arbitrary. Then equation~\eqref{eqn:kinetic} has a solution $\Epsf \in C([0, T]; C^1(\Rd \times \Rd))$ with initial data  ${\Epsf}\big|_{t=0} \, = f_0$. 
\end{Theorem}
\begin{proof}[Sketch of Proof] 

The proof is divided in three steps: construct an approximating sequence $f^{(n)}_\Eps$ in $C([0, T]; C^2 (\Rd \times \Rd))$ by iterations, prove a uniform bound of $f^{(n)}_\Eps$ in $C([0, T]; C^1(\Rd \times \Rd))$, and show that $f^{(n)}_\Eps$ is a Cauchy sequence in $C([0, T]; C^1(\Rd \times \Rd))$ which converges to the desired solution of \eqref{eqn:kinetic}. 
\end{proof}

We end this section by pointing out that, similar to the analysis in \cite{CanizoCarrilloRosado2011} for measure-valued solutions, the classical results invoked to prove Theorem \ref{thm:exist-smooth} also use the characteristic equations \eqref{eq:characteristics-eps}. For smooth solutions however the mass transportation formula \eqref{eqn:mass-transport} is equivalent to solving the equation by the method of characteristics.


\section{Uniform in $\Eps$ estimates}
\label{sect:estimates}
We present in this section all the (uniform in $\Eps$) estimates needed to prove the convergence as $\Eps \to 0$ of solutions to  \eqref{eqn:kinetic}. Throughout the rest of the paper we assume $\nabla K \in W^{1,\infty}(\Rd)$ and that the initial density $f_0$ has compact support. Note that compared to \cite{CanizoCarrilloRosado2011} we require a slightly stronger condition on $K$, as for our analysis we need a global Lipschitz bound to obtain the uniform bound in $\Eps$ for the support of the solution (Proposition \ref{prop:supp}). 

For further reference, let us write the initial-value problem for \eqref{eqn:kinetic}, with explicit $\Eps$-dependence indicated for its solution $f_\Eps$:
\begin{equation}
   \label{eq:kinetic}
\begin{aligned}
    \del_t \Epsf + v \cdot \Grad \Epsf
   & \, = \frac{1}{\Eps} \Gradv \cdot \vpran{\vpran{v + \Grad K \ast \Epsrho} \Epsf},
\\
    \Epsf \big|_{t=0} &\, = f_0 (x, v),
\end{aligned}
\end{equation}
where
\[
\rho_\epsilon(t,x) = \int_{\mathbb{R}^d} f_\epsilon(t,x,v) dv.
\]


\paragraph{Solutions of compact support.} We first make the observation that  $f_0$ being compactly supported implies that solutions $\Epsf$ (either smooth or measure-valued) remain compactly supported for all times (see Theorem \ref{thm:exist-measure}). We show below that the support of $\Epsf$ is in fact {\em independent} of $\Eps$.

\begin{Proposition}[Uniform estimate for the support]
\label{prop:supp} 
Consider a solution $\Epsf$ to \eqref{eq:kinetic} as provided by Theorem \ref{thm:exist-measure}. Then, there exists an increasing function $R(T)$ (independent of $\Eps$) such that for all $T>0$,
\begin{equation}
\label{eqn:suppb}
\operatorname{supp }\Epsf(t) \subset B_{R(T)} \quad  \text{ for all } t \in [0,T] \text{ and } \Eps>0.
\end{equation}
The function $R(T)$ depends only on the support of $f_0$ and $\| \nabla K \|_{L^\infty}$. 
\end{Proposition}
\begin{proof}
The support of $\Epsf$ evolves with the flow governed by the characteristic equations \eqref{eq:characteristics-eps},  initialized at points $(x_0,v_0)$ in the support of $f_0$. Since
\[
|\nabla K \ast \Epsrho| \leq \| \nabla K \|_{L^\infty},
\]
from \eqref{eq:characteristics-eps} we infer that the Euclidean norm $|v_\Eps(t)|$ of the $v$-trajectories satisfies
\[
\frac{d|v_\Eps|}{dt} \leq -\frac{1}{\Eps} |v_\Eps| + \frac{1}{\Eps}  \| \nabla K \|_{L^\infty}, \qquad v_\Eps(0)=v_0.
\]
Consequently, there exists a constant $C$ that depends only on the support of $f_0$ and $\| \nabla K \|_{L^\infty}$, such that all characteristic paths that start from within $\operatorname{supp} f_0$ satisfy
\[
|v_\Eps(t)| \leq C, \quad \text{ for all } t>0 \text{ and } \Eps>0. 
\]
Since $\frac{d|x_\Eps|}{dt} \leq |v_\Eps|$, the $x$-trajectories grow at most linearly in time. Hence there exists a function $R(T)$ that depends only on the support of $f_0$, $\| \nabla K \|_{L^\infty}$ and $T$ such that \eqref{eqn:suppb} holds.
\end{proof}


\subsection{Estimates for smooth solutions} 
\label{subsect:est-smooth}
Consider the smooth case and take solutions $\Epsf \in C([0, T]; C^1(\Rd \times \Rd))$, as provided by Theorem \ref{thm:exist-smooth}. 
The key estimate needed for the convergence is provided by the following result.
\begin{Proposition}[Main estimate for smooth solutions]
\label{thm:estimate}
  Let $\Epsf$ be the classical solution to~\eqref{eq:kinetic}, as provided by Theorem \ref{thm:exist-smooth}. Assume additionally that the initial data $f_0$ has a finite first moment in v, that is, $|v| f_0 \in L^1(\Rd \times \Rd)$. Then
there exists a constant $C_0, \Eps_0$ such that for any $\Eps \leq \Eps_0$,
\begin{align} \label{bound:generalized-moment}
    \iint_{\R^d \times \R^d} \abs{v + \Grad K \ast \Epsrho} \Epsf\dx\dv
\leq C_0 \Eps, \quad \text{ for all } t\in [0,T],
\end{align}
where $C_0, \Eps_0$ only depends on $\norm{\Grad K}_{W^{1, \infty}}$ and $\norm{(1 + |v|) f_0}_{L^1(\Rd \times \Rd)}$. 
\end{Proposition}
\begin{proof}
Denote the quantity on the left-hand side of~\eqref{bound:generalized-moment} as
\begin{align*}
   I(t) = \iint_{\R^d \times \R^d} \abs{v + \Grad K \ast \Epsrho} \Epsf\dx\dv.
\end{align*}
Hence our goal is to show that there exists $C_0$ such that
\begin{align} \label{eqn:generalized-moment}
   \sup_{t \in [0, T]} I(t) \leq C_0 \Eps.
\end{align}
Multiply equation~\eqref{eq:kinetic} by $\abs{v + \Grad K \ast \Epsrho}$ and integrate in $x, v$ to get
\begin{equation}
\label{eqn:boundR}
   \frac{\rm d}{\dt} I(t) 
 =  - \frac{1}{\Eps} I(t) 
 + \underbrace{\iint_{\R^d \times \R^d} \vpran{\del_t \abs{v + \Grad K \ast \Epsrho}}
                                         \Epsf \dx\dv}_{=I_1(t)}
 + \underbrace{\iint_{\R^d \times \R^d} 
            \vpran{v \cdot \Grad \abs{v + \Grad K \ast \Epsrho}} \Epsf \dx\dv}_{=I_2(t)}.                                      
\end{equation}
Denote the two remainder terms in the right-side as 
\begin{align*}
   I_1(t) & \, = \iint_{\R^d \times \R^d} \vpran{\del_t \abs{v + \Grad K \ast \Epsrho}} \Epsf \dx\dv,
\\
   I_2(t) & \, = \iint_{\R^d \times \R^d} 
            \vpran{v \cdot \Grad \abs{v + \Grad K \ast \Epsrho}} \Epsf \dx\dv.
\end{align*}
The strategy is to show that $I_1(t)$ and $I_2(t)$ are bounded linearly by $I(t)$ and then derive a differential inequality from \eqref{eqn:boundR} to bound $I(t)$. 

By integrating \eqref{eq:kinetic} in $v$ one finds that the macroscopic density $\Epsrho \in C([0, T]; C^1(\R^d))$ satisfies 
\begin{equation}
\label{eq:rhoeps}
\begin{aligned}
   \del_t \Epsrho + \Grad \cdot & \, \vint{ v \Epsf} = 0, 
\\
   \Epsrho \big|_{t = 0} &= \vint{f_{0}}.
\end{aligned}
\end{equation}
Here, angle brackets denote integration with respect to $v$. Equation \eqref{eq:rhoeps} conserves mass:
\[
\norm{\Epsrho(t)}_{L^1(\Rd)} = \norm{f_0}_{L^1(\Rd \times \Rd)}, \qquad \text{ for all } t>0.
\]

Hence,
\begin{align*}
    \big| \del_t \abs{v + \Grad K \ast \Epsrho} \big|
& \leq 
    \abs{\Grad K \ast \del_t \Epsrho}
= \abs{\Grad^2 K \ast \vint{v \Epsf}}
\\
& \leq
  \abs{\Grad^2 K \ast \vint{(v + \Grad K \ast \Epsrho) \Epsf}}
  + \norm{\Grad^2 K \ast \Epsrho}_{L^\infty} \norm{\Grad K \ast \Epsrho}_{L^\infty}
\\
& \leq
  \norm{\Grad^2 K}_{L^\infty} I(t)
  + \norm{\Grad^2 K}_{L^\infty} 
      \norm{\Grad K}_{L^\infty}    
      \norm{\Epsrho}_{L^1(\Rd)}^2.
\end{align*}
Therefore, 
\begin{align}
   \abs{I_1(t)}
& \leq 
   \norm{\Grad^2 K}_{L^\infty} \norm{\Epsrho}_{L^1(\Rd)} I(t) 
   + \norm{\Grad^2 K}_{L^\infty} 
      \norm{\Grad K}_{L^\infty}    
      \norm{\Epsrho}_{L^1(\Rd)}^3 \nonumber
\\
& \leq
   C_{1} I(t) + C_{2}, \label{eqn:estR1}
\end{align}
where $C_{1}, C_{2}$ only depend on $\norm{\Grad K}_{W^{1, \infty}}$ and $\norm{f_0}_{L^1(\Rd \times \Rd)}$. 

A similar estimate can be derived for $I_2(t)$. Indeed, 
\begin{align}
   |I_2(t)| 
& \leq 
   \iint_{\R^d \times \R^d}
      |v| \abs{\Grad^2 K \ast \Epsrho} \Epsf \dx\dv \nn
\\
& \leq
   \norm{\Grad^2 K}_{L^\infty} \norm{\Epsrho}_{L^1(\Rd)}
   \iint_{\R^d \times \R^d}
      |v| \Epsf \dx\dv \label{eqn:estR2}
\\
& \leq
   C_{3}  \Biggl(  \;\,  \iint_{\R^d \times \R^d}
      |v + \Grad K * \Epsrho| \Epsf \dx\dv
      + \norm{\Grad K}_{L^\infty} \norm{\Epsrho}_{L^1(\Rd)}^2 \Biggr) \nonumber
\\
& \leq
   C_{3} I(t)  + C_{4}, \nn
\end{align}
where $C_{3}, C_{4}$ only depend on $\norm{\Grad K}_{W^{1, \infty}}$ and $\norm{f_0}_{L^1(\Rd \times \Rd)}$. 

Combining \eqref{eqn:boundR}, \eqref{eqn:estR1} and \eqref{eqn:estR2} we obtain the following differential inequality for $I(t)$: 
\begin{align} \label{eq:R}
   \frac{\rm d}{\dt} I(t) 
 \leq & \, - \frac{1}{\Eps} I(t) 
 + C_{5} I(t) + C_{6},                                 
\end{align}
where $C_{5} = C_{1} + C_{3}$ and $C_{6} = C_{2} + C_{4}$, both of which depending only on $\norm{\Grad K}_{W^{1, \infty}}$ and $\norm{f_0}_{L^1(\Rd \times \Rd)}$. 

Note that initially,
\begin{equation}
\begin{aligned} \label{cond:R-initial}
   I(0) 
& \leq \norm{|v| f_0}_{L^1(\Rd \times \Rd)} + C_{7},
\end{aligned}
\end{equation}
where $C_{7}$ only depends on $\norm{\Grad K}_{L^\infty}$ and $\norm{f_0}_{L^1(\Rd \times \Rd)}$.  Finally, from ~\eqref{eq:R} and \eqref{cond:R-initial} one can derive
\begin{align*}
   \sup_{t \in [0, T]} I(t) \leq C_0 \Eps,
\end{align*}
where $C_0$ only depends on $\norm{\Grad K}_{W^{1, \infty}}$ and $\norm{(1 + |v|) f_0}_{L^1(\Rd \times \Rd)}$. 
\end{proof}


\subsection{Estimates for measure-valued solutions}
\label{subsect:est-measure}

Next, we show that a similar bound as in~\eqref{bound:generalized-moment} holds for a measure-valued solution $\Epsf$ as well. The strategy we employ here is to use a smooth approximating sequence for which the results in Section \ref{subsect:est-smooth} are valid, and then pass  to the limit to infer results for measure solutions.

Take an initial measure $f_0 \in \mathcal{P}_1(\Rd \times \Rd)$ with compact support and fix $\Eps>0$.  Let $f_{0}^{(n)}$ be a sequence of mollifications of $f_0$ such that
\begin{equation}
\label{eqn:f0moll}
    f_{0}^{(n)} = f_0 \ast \eta^{(n)} \in C^2(\Rd \times \Rd).
\end{equation}
Here the mollifier can be chosen such that $\eta^{(n)}(x, v) = n^{2d}\eta^{(1)}(nx, nv) \in C^\infty_c(\R^d \times \R^d)$, where $\eta^{(1)}$ is compactly supported over the unit ball in $\R^{2d}$ and satisfies
\begin{align*}
   \eta^{(1)} \geq 0,
\qquad
    \iint_{\R^d \times \R^d} \eta^{(1)}(x, v) \dx\dv = 1,
\qquad
    \iint_{\R^d \times \R^d} |v| \eta^{(1)}(x, v) \dx\dv \leq 1.
\end{align*}

The following mollification lemma is classical (see for example \cite{Amb2003}).
\begin{Lemma}
\label{lemma:f0n}
Suppose $f_0 \in \CalP_1(\R^d \times \R^d)$ with $supp \, f_0 \subseteq B(R_0)$. Then the approximating sequence $f_{0}^{(n)}$ satisfies
\begin{itemize}
\item[(a)] $supp \, f_0^{(n)} \subseteq B(R_0+1)$ for all $n \geq 1$.
\item[(b)] $f_{0}^{(n)} \in \CalP_1(\Rd \times \Rd)$ and the first moments $\iint  |v| f_{0}^{(n)}(x,v) dx dv$ are uniformly bounded.
\item[(c)] $\{f_{0}^{(n)}\}_{n\geq 1}$ is a Cauchy sequence in $\CalP_1(\Rd \times \Rd)$ endowed with the $W_1$ distance, and $\| f_{0}^{(n)}-f_0 \|_{W_1} \to 0$  as $n\to \infty$.
\end{itemize}
\end{Lemma}

Now we construct the approximating sequence $\Epsfn$ such that
\begin{equation}
\begin{aligned} \label{eq:transport-eps-n}
    \del_t \Epsfn + v \cdot \Grad \Epsfn
   & \, = \frac{1}{\Eps} \Gradv \cdot \vpran{\vpran{v + \Grad K \ast \Epsrhon} \Epsfn},
\\
    \Epsfn \big|_{t=0} &\, = f_{0}^{(n)} (x, v),
\end{aligned}
\end{equation}
where $\Epsrhon = \int_{\R^d} \Epsfn \dv$. 

\begin{Lemma} \label{lem:f-n-Eps}
Fix $\Eps > 0$. Suppose $\Grad K \in W^{1, \infty}(\Rd)$ and $f_0 \in \CalP_1(\R^d \times \R^d)$ with compact support. Let $f_0^{(n)}$ be the sequence of mollifications of $f_0$ given by~\eqref{eqn:f0moll}. Then for each $T > 0$, there exists a sequence of solutions $\Epsfn \in C([0, T); C^1(\R^d \times \R^d))$ to~\eqref{eq:transport-eps-n} whose supports only depend on $T$ and $\norm{\Grad K}_{W^{1, \infty}}$. In particular, the supports are uniformly bounded in both $n$ and $\Eps$. 

Moreover,  if we let $f_\Eps \in C([0, T), \CalP_1(\R^d \times \R^d))$ be the unique measure solution to~\eqref{eq:kinetic} in the sense of Definition~\ref{defn:sol}, then
\begin{equation}
 \label{convergence:1}
  \Epsfn(t, \cdot, \cdot) \stackrel{W_1}{\longrightarrow} \Epsf(t, \cdot, \cdot) 
\quad \text{ in }  \CalP_1(\Rd \times \Rd), \quad \text{ uniformly in } t \text{ as } n \to \infty.
\end{equation}
\end{Lemma}
\begin{proof}
Since $f_0^{(n)} \in C^2(\R^d \times \R^d)$ has compact support, 
we can apply the existence theory in Theorem \ref{thm:exist-smooth} and deduce that there exists a smooth solution $\Epsfn \in C([0, T]; C^1(\R^d \times \R^d))$  of \eqref{eq:transport-eps-n} for every $n \geq 1$ and every $\Eps > 0$. Each $\Epsfn$ is compactly supported and integrates to $1$. 
Proposition~\ref{prop:supp} yields that the support of $\Epsfn$ is independent of $\Eps$ and depends only on $T$, $\norm{\Grad K}_{W^{1, \infty}}$, and the support of $f_0^{(n)}$. 
By part (a) in Lemma~\ref{lemma:f0n}, we further conclude that the supports of $\Epsfn(t, \cdot, \cdot)$ are uniformly bounded in {\em both} $n$ and $\Eps$ for all $t \in [0,T]$.

Let $f_\Eps \in C([0, T); \CalP_1(\R^d \times \R^d))$ be the unique measure-valued solution to~\eqref{eq:kinetic} in the sense of Definition \ref{defn:sol}. By the stability result in \cite[Theorem 3.16]{CanizoCarrilloRosado2011}, we have that for all times $t \geq0$,
\begin{align} \label{bound:stability}
    \bigl \|\Epsfn(t, \cdot, \cdot) - f_\Eps(t, \cdot, \cdot) \bigr \|_{W_1} 
\leq
    r(T) \bigl \| f_0^{(n)} - f_0 \bigr \|_{W_1},
  \end{align}
where $r(T)$ 
only depends on $T$ and the support of $f_0$. 
From  \eqref{bound:stability} and Lemma \ref{lemma:f0n} we then conclude \eqref{convergence:1}. 
%
%
\end{proof}

It follows from \eqref{convergence:1} that 
\begin{equation}
\label{convergence:2}
 \rho_\Eps^{(n)}(t, \cdot, \cdot) \longrightarrow \rho_\Eps(t, \cdot, \cdot) \quad \text{ as measures}, \quad \text{ for each } t \in [0, T) \text{ as } n\to \infty,
\end{equation}
where $\rho_\Eps$ is the first marginal of $f_\Eps$. Next we show the following $L^\infty$-convergence of $\Grad K \ast \Epsrhon$.
\begin{Lemma} 
\label{lemma:1} 
Consider the measure solution $\Epsf$ of \eqref{eq:kinetic} obtained as a limit of mollifications $\Epsfn$ as in Lemma~\ref{lem:f-n-Eps}.
Then for each $t \geq 0$, we have
$\Grad K \ast \Epsrho\in C(\R^d)$ and 
\begin{align*}
    \Grad K \ast \Epsrhon(t, \cdot) \longrightarrow \Grad K \ast \rho_\Eps(t, \cdot)
\qquad \text{strongly in } L^\infty_{loc}(\R^d), \text{ as } n \to \infty.
\end{align*}
\end{Lemma}
\begin{proof}
Given the regularity of $\Epsfn$ we have that $\Grad K \ast \Epsrhon(t,\cdot) \in C(\R^d)$ for each time $t \geq 0$. Also, by mass conservation of $\Epsfn$  and the properties of the mollifiers in Lemma~\ref{lemma:f0n}, $\Grad K \ast \Epsrhon$ satisfies 
\[
   |\Grad K \ast \Epsrhon(x) | \leq  \norm{\Grad K}_{L^\infty},
\]
\[
   |\Grad K  \ast \Epsrhon (x_1) - \Grad K  \ast \Epsrhon(x_2)| \leq \norm{\Grad^2 K}_{L^\infty} |x_1 - x_2|,
\]
for any $x, x_1, x_2 \in \R^d$. Hence the sequence $\{ \Grad K \ast \Epsrhon \}_{n \geq 1}$ is uniformly bounded and equi-continuous.
By Ascoli-Arzel\'{a} theorem, we have that $\{ \Grad K \ast \Epsrhon \}_{n\geq 1}$ converges on a subsequence in the strong topology of $C(\Omega)$, for any compact set $\Omega \subset\R^d$. Meanwhile, by \eqref{convergence:2}, the limit function is $\Grad K \ast \rho_\Eps \in C(\Rd)$. 
It then follows that the entire sequence $\{ \Grad K \ast \Epsrhon \}_{n\geq 1}$ converges to $\Grad K \ast \rho_\Eps$, as desired.
\end{proof}

Now we state the analogue for measure solutions of the main estimate \eqref{bound:generalized-moment}.
\begin{Proposition}[Main estimate for measure-valued solutions]
\label{thm:bound-f-eps}
Fix $\Eps > 0$ and assume the hypotheses in Lemma \ref{lem:f-n-Eps}. 
Then for any $\tilde\phi \in C_b(\R^d \times \R^d)$,
there exists a constant $\tilde{C}_0$ such that 
\begin{align} \label{bound:generalized-moment-f-eps}
    \Biggl | \; \,  \iint_{\R^d \times \R^d} \tilde\phi(x, v) \vpran{v + \Grad K \ast \Epsrho} \Epsf\dx\dv \Biggr |
\leq \tilde{C}_0 \Eps, \quad \text{ for all } t \in [0,T].
\end{align}
Specifically, $\tilde{C}_0 = C_{0} \| \tilde\phi \|_{L^\infty(\R^d \times \R^d)}$, where $C_{0}$ is a constant which depends only on $\norm{\Grad K}_{W^{1, \infty}}$ and $\iint |v| f_0 \, dx dv$. 
In particular, $\tilde C_0$ is independent of $\Eps$ and $t$.
\end{Proposition}
\begin{proof}
Let $\Epsfn$ be the approximating sequence satisfying~\eqref{eq:transport-eps-n}. Hence $\Epsfn$ satisfies~\eqref{bound:generalized-moment} and for any $\tilde\phi \in C_b(\R^d \times \R^d)$,
\begin{align}  \label{bound:generalized-moment-1}
    \Bigl | \iint_{\R^d \times \R^d} \tilde\phi(x, v) \vpran{v + \Grad K \ast \Epsrhon} \Epsfn \dx\dv \Bigr |
\leq C_{0} \| \tilde\phi \|_{L^\infty} \Eps.
\end{align}
Note that by Lemma \ref{lemma:f0n}, the constant $C_{0}$ can be chosen to be independent of $n$, and depending only on $\norm{\Grad K}_{W^{1, \infty}}$ and $\iint |v| f_0 \, dx dv$.

Denote by $\Omega(T) \subset \R^{2d}$ the common support of $\Epsfn(t)$ for all $\Eps>0$, $n \geq 1$ at any $t \in [0, T]$. Then, by~\eqref{convergence:1} and Lemma~\ref{lemma:1}, we have that for each $t\in [0,T]$,
\begin{align*}
& 
   \iint_{\R^d \times \R^d} \tilde\phi(x, v) \vpran{v + \Grad K \ast \Epsrhon} \Epsfn\dx\dv
\\
&= \iint_{\Omega(T)} \tilde\phi(x, v) \vpran{v + \Grad K \ast \Epsrhon} \Epsfn \dx\dv
\rightarrow
   \iint_{\R^d \times \R^d} \tilde\phi(x, v) \vpran{v + \Grad K \ast \Epsrho} \Epsf\dx\dv,
\end{align*}
as $n \to \infty$. Hence by taking the limit  $n\to \infty$ in~\eqref{bound:generalized-moment-1}, we obtain the desired bound in~\eqref{bound:generalized-moment-f-eps}. 
 \end{proof}


\section{Convergence as $\Eps \to 0$ of solutions to \eqref{eqn:kinetic}}
\label{sect:conv}

By Proposition~\ref{thm:bound-f-eps}, measure-valued solutions $\Epsf$  of the transport equation~\eqref{eq:kinetic}, in the sense of Definition \ref{defn:sol}, satisfy the uniform (in $\Eps$) estimate \eqref{bound:generalized-moment-f-eps}. In this section we use this key estimate to pass the limit $\Eps \to 0$ in \eqref{eq:kinetic}, that is, in the initial-value problem for \eqref{eqn:kinetic}.

First we explain the setting for well-posedness of the macroscopic equation~\eqref{eq:aggre}. Consider the initial value problem for \eqref{eq:aggre}:
\begin{equation}
\begin{aligned} \label{eq:rho}
    \rho_t - \Grad \cdot &((\Grad K \ast \rho) \rho) =  0 \,, 
\\
    \rho \big|_{t=0} &= \rho_0 (x) \,.
\end{aligned}
\end{equation}
Similar to the kinetic equation, there exist several concepts of solutions to \eqref{eq:rho} over an arbitrary time interval $[0, T)$. 

The first concept is the measure-valued solution as defined in \cite{CanizoCarrilloRosado2011}. More specifically, assuming that  $\Grad K \in W^{1, \infty}(\Rd)$, one can apply the framework in \cite{CanizoCarrilloRosado2011} and obtain a unique measure-valued solution $\rho \in C([0, T); \CalP_1(\R^d))$ in the mass transportation sense (similar to how a measure solution for the kinetic equation \eqref{eqn:kinetic} has been introduced in Definition~\ref{defn:sol}).

The second notion is the weak solution in $C([0, T); \CalP(\R^d))$ where the continuity in time is in the narrow sense. In particular, a weak solution $\rho \in C([0, T); \CalP(\R^d))$ to~\eqref{eq:rho} satisfies  
\begin{align} \label{eq:rho-weak}
   \int_0^T \int_{\R^d} \del_t\psi(t, x) \rho(t,x) \dx\dt
   - \int_0^T \int_{\R^d} \nabla_x \psi \cdot \vpran{\Grad K \ast \rho} \rho \dx\dt
   + \int_{\R^d} \psi(0, x) \rho_0(x) \dx =0,
\end{align}
for any $\psi \in C^1_c([0, T); C^1_b(\R^d))$. A global-in-time well-posedness theory of weak measure solutions to \eqref{eq:rho} was established in \cite{Figalli_etal2011} for a very general class of (nonsmooth) potentials. In their setting, as well as ours, the two concepts of solutions are in fact equivalent (see Step 3 in the proof of Theorem \ref{thm:conv-sol} for a more detailed account of this fact).


\subsection{Convergence to the macroscopic equation}
\label{subsect:conv-eq}

The main result is the following theorem.
\begin{Theorem}
\label{thm:conv-sol}
Let $T > 0$ be arbitrary, $\Grad K \in W^{1, \infty}(\R^d)$ and $f_0 \in \CalP_1(\R^d \times \R^d)$ with compact support. Suppose $\Epsf \in C([0, T); \CalP_1(\R \times \Rd))$ is the measure-valued solution to~\eqref{eq:kinetic} obtained in Theorem \ref{thm:exist-measure}. Let $\rho_\Eps$ be the first marginal of $\Epsf$ as defined in \eqref{eqn:1stm}. 

Then, there exists $\rho \in C([0, T); \CalP_1(\R^d))$ such that for each $t \in [0, T)$,
\begin{align} \label{convg:rho-Eps}
     \Epsrho(t, \cdot) \stackrel{W_1}{\longrightarrow} \rho(t, \cdot)
\qquad \text{in $\CalP_1(\R^d)$ as $\Eps \to 0$.}
\end{align}
Moreover, $\rho$ is the unique solution to the initial value problem \eqref{eq:rho} in the weak sense, i.e., it satisfies \eqref{eq:rho-weak}, where $\rho_0$ is the first marginal of $f_0$. 
\end{Theorem}


\begin{proof}
We break the proof into several steps.

{\em Step 1.} First we show a  uniform in time convergence of $\rho_\Eps$. Recall that the measure solution $\Epsf$ of \eqref{eq:kinetic}  is also a weak solution, cf. \eqref{eq:f-Eps-weak}. For any fixed $\psi_1 \in C_c^1(0, T)$ and $\psi_2 \in C^1_b(\R^d)$, let $\phi(t, x, v) = \psi_1(t)\psi_2(x)$ and use it as a test function in \eqref{eq:f-Eps-weak}, to get:
\begin{equation}
\label{eqn:rhoeps-weak-1}
    \int_0^T \psi'_1(t)\int_{\R^{d}} \psi_2(x) \Epsrho(t, x) \dx\dt
    = - \int_0^T \psi_1(t) \iint_{\R^d \times \R^d} \Grad \psi_2 \cdot v \Epsf \dx\dv\dt.
\end{equation}
Denote 
\begin{equation}
\label{eqn:eta1}
\eta_1(t) = \int_{\R^{d}} \psi_2(x) \Epsrho(t, x) \dx.
\end{equation}
By \eqref{eqn:rhoeps-weak-1}, the weak derivative of $\eta_1$ is given by
\begin{align*}    
    \eta'_1(t) = \iint_{\R^d \times \R^d} \Grad \psi_2 \cdot v \Epsf \dx\dv
    \in L^\infty(0, T).
\end{align*}
Since $f_\Eps$ is uniformly supported on $\Omega(T)$, we have
\begin{align} \label{bound:eta-1}
    \norm{\eta_1}_{W^{1, \infty}(0, T)}
\leq C(T)\norm{\psi_2}_{C^1_b(\R^d)},
\end{align}
where $C(T)$ only depends on $T$ (in particular, $C(T)$ is independent of $\Eps$).

Since $\eta_1(t)$ is uniformly bounded in $W^{1, \infty}(0, T)$, it converges uniformly on a subsequence.
We conclude that given any $\psi_2 \in C^1_b(\R^d)$, there exists a subsequence $\Eps_k$ and $\eta_2(t) \in C([0, T))$ such that 
\begin{equation}
\label{eqn:convk}
    \int_{\R^d} \psi_2(x) \rho_{\Eps_k}(t, x) \dx 
\to \eta_2(t) \quad \text{uniformly in } C([0, T)) \quad \text{ as } \Eps_k \to 0.
\end{equation}

On the other hand, note that Proposition \ref{prop:supp} provides a uniform (in $\Eps$) bound for the support of $\Epsrho$, which implies that the sequence $\Epsrho(\cdot,t)$  is tight. 
By Prokhorov's theorem (cf. \cite[Theorem 4.1]{Billingsley}), for each $t \in [0, T)$, $\Epsrho(t,\cdot)$ converges weak-$^\ast$ on a subsequence to a probability measure $\rho(\cdot,t) \in \CalP(\R^d)$. By Remark~\ref{rmk:W1}, the convergence holds in fact in $\mathcal{P}_1(\Rd)$, with respect to the Wasserstein metric $W_1$. 

Hence, at each $t>0$, there exist a subsequence of $\rho_{\Eps_k}$, denoted as $\rho_{\Eps_{k_n}}$ ($k_n$ may depend on $t$), which satisfies
\begin{align*} 
     \rho_{\Eps_{k_n}} (t, \cdot)  \stackrel{W_1}{\longrightarrow} \rho(t, \cdot)
\qquad \text{ in } \CalP_1(\R^d) \quad \text{ as } \Eps_{k_n} \to 0.
\end{align*}
Consequently, for each $t \in [0, T)$ and any $\psi_2 \in C^1_b(\R^d)$,
\begin{align*}
    \int_{\R^d} \psi_2(x) \rho_{\Eps_{k_n}}(t, x) \dx 
\to \int_{\R^d} \psi_2(x) \rho(t, x) \dx \qquad \text{as } \Eps_{k_n} \to 0.
\end{align*}
Combined with \eqref{eqn:convk} and the uniqueness of $\eta_2(t)$ at each $t \in [0, T)$, this shows that the full sequence $\rho_{\Eps_k}(t, \cdot)$ (with $\Eps_k$ independent of $t$) and $\rho(t, \cdot) \in \CalP_1(\R^d)$ satisfy
\begin{equation}
\label{eqn:convku}
    \int_{\R^d} \psi_2(x) \rho_{\Eps_k}(t, x) \dx 
\to \int_{\R^d} \psi_2(x) \rho(t, x) \dx \qquad \text{uniformly on } [0, T) \quad \text{ as } \Eps_k \to 0,
\end{equation}
for any $\psi_2 \in C^1_b(\R^d)$. In addition, , we have that 
\begin{align} \label{convg:6}
     \rho_{\Eps_k} (t, \cdot)  \stackrel{W_1}{\longrightarrow} \rho(t, \cdot)
\qquad \text{in } \CalP_1(\R^d) \quad \text { as } \Eps_k \to 0.
\end{align}

Next we show that we can also allow $\psi_2$ to depend on $t$ in \eqref{eqn:convku}. Specifically, we claim that given any $\psi_3 \in C_c([0, T); C^1_b(\R^d))$, 
\begin{align} \label{bound:narrow-unif-rho-Eps-1}
    \int_{\R^d} \psi_3(t, x) \rho_{\Eps_k}(t, x) \dx 
\quad \text{is equicontinous on $[0, T)$ .}
\end{align}
Indeed, for any $t, s \in [0, T)$, 
\begin{align*}
& \quad \,
   \biggl |\, \int_{\R^d} \psi_3(t, x) \rho_{\Eps_k}(t, x) \dx 
           -  \int_{\R^d} \psi_3(s, x) \rho_{\Eps_k}(s, x) \dx \biggr |
\\
&\leq \int_{\R^d} \abs{\psi_3(t, x) - \psi_3(s, x)} \rho_{\Eps_k}(t, x) \dx 
       + \biggl | \, \int_{\R^d} \psi_3(s, x) \rho_{\Eps_k}(t, x) \dx 
           -  \int_{\R^d} \psi_3(s, x) \rho_{\Eps_k}(s, x) \dx \biggr |.
\\
& \leq \sup_{x} \abs{\psi_3(t, x) - \psi_3(s, x)}
          + C(T)\sup_t \norm{\psi_3}_{C^1_b(\R^d)} |t-s| \,,
\end{align*}
where the last inequality follows from~\eqref{bound:eta-1}. Since $\psi_3$ is uniformly continuous on $[0, T) \times \R^d$, we have that 
\begin{align*}
    \sup_{x} \abs{\psi_3(t, x) - \psi_3(s, x)} \to 0 \,,
\qquad \text{uniformly as $|t-s| \to 0$.}
\end{align*}
This shows that \eqref{bound:narrow-unif-rho-Eps-1} holds. Hence up to a subsequence, still denoted as $\rho_{\Eps_k}$, we have
\begin{align} \label{bound:narrow-unif-rho-Eps-2}
    \int_{\R^d} \psi_3(t, x) \rho_{\Eps_k}(t, x) \dx  
\to \int_{\R^d} \psi_3(t, x) \rho(t, x) \dx
\qquad \text{uniformly on } [0, T) \quad \text{ as } \Eps_k \to 0,
\end{align}
for any $\psi_3 \in C_c([0, T); C^1_b(\R^d))$. 
\smallskip

{\em Step 2.} In this step we pass the limit $\Eps_k \to 0$ on the subsequence $\rho_{\Eps_k}$ to find a limiting equation for $\rho$.

Define
\begin{align} \label{def:Omega-1}
   \Omega_1(T) = \{x \in \R^d: (x, v) \in \Omega(T)\} \,,
\end{align} 
where recall that $\Omega(T) \subset \R^{2d}$ represents the common support of $\Epsf(t)$ for all $\Eps>0$ and $t \in [0, T]$. We have that $\Omega_1(T)$ is bounded and that the supports of $\rho_{\Eps_k}$ and $\rho$ are included in $\Omega_1(T)$ for all $t \in [0, T]$. 

For any $\psi \in C^1_c([0, T); C_b^1(\R^d))$, let $\phi(t, x, v) = \psi(t, x)$ in~\eqref{eq:f-Eps-weak}. Then
\begin{equation}
\label{eqn:rhoeps-weak}
    \int_0^T\int_{\R^{d}} \del_t\psi(t,x) \rho_{\Eps_k}(t, x) \dx \dt
    +\int_0^T \iint_{\R^d \times \R^d} \Grad \psi \cdot v f_{\Eps_k} \dx\dv\dt
   + \int_{\R^d} \psi(0, x) \rho_0(x) \dx =0.
\end{equation}

We want to pass $\Eps_k \to 0$ in \eqref{eqn:rhoeps-weak}. By~\eqref{bound:narrow-unif-rho-Eps-2}, 
\begin{align*}
     \int_0^T\int_{\R^{d}} \del_t\psi(t,x) \rho_{\Eps_k}(t, x) \dx \dt
\to 
    \int_0^T\int_{\R^{d}} \del_t\psi(t,x) \rho(t, x) \dx \dt
\qquad \text{as $\Eps_k \to 0$}.
\end{align*}

Next we rewrite the integrand of the second term in~\eqref{eqn:rhoeps-weak}  as
\begin{equation}
\label{eqn:rewrite}
    \iint_{\R^d \times \R^d} \Grad \psi \cdot v f_{\Eps_k} \dx\dv
  = \iint_{\R^d \times \R^d} 
        \Grad \psi \cdot \vpran{v + \Grad K \ast \rho_{\Eps_k}} f_{\Eps_k} \dx\dv
     - \int_{\R^d}
        \Grad \psi \cdot \vpran{\Grad K \ast \rho_{\Eps_k}} \rho_{\Eps_k} \dx.
\end{equation}
Use \eqref{bound:generalized-moment-f-eps} to get
\begin{align} \label{convg:4}
    \iint_{\R^d \times \R^d} 
        \Grad \psi \cdot \vpran{v + \Grad K \ast \rho_{\Eps_k}} f_{\Eps_k} \dx\dv
    \to 0 
\qquad \text{as } \Eps_k \to 0 \quad \text{ uniformly in } t.
\end{align}

By the same argument as in Lemma~\ref{lemma:1}, one can show that $\{\Grad K \ast \rho_{\Eps_k}(t, \cdot)\}$ is a bounded family in $W^{1, \infty}(\R^d)$ for each $t \in [0, T)$. More precisely,
\begin{align*}
     \norm{\Grad K \ast \rho_{\Eps_k}(t, \cdot)}_{W^{1, \infty}(\Rd)}
\leq \norm{\Grad K}_{W^{1, \infty}}.
\end{align*}
Now we want to show that $\Grad K \ast \rho_{\Eps_k}$ is also equicontinuous in $t$. Note that $\Grad K$ does not have enough regularity for the bound in~\eqref{bound:eta-1} to apply directly. To bypass  this, we mollify $K$ by convolution and let $K_n = K \ast \eta^{(n)}$ where $\eta^{(n)}$ is the same mollifier defined in~\eqref{eqn:f0moll}. Hence $\Grad K_n = \Grad K \ast \eta^{(n)}$ and $\Grad^2 K_n = \Grad^2 K \ast \eta^{(n)}$. This shows
\begin{align*}
     \norm{\Grad K_n}_{C^1_b} 
\leq 
    \norm{\Grad K}_{W^{1, \infty}}
\qquad \text{for all $n \geq 1$}.
\end{align*}
Use $\psi_2 = \Grad K_n$ in \eqref{eqn:eta1}. Then, bound \eqref{bound:eta-1} yields
\begin{align*}
    \sup_{x}\norm{\Grad K_n \ast \rho_{\Eps_k}}_{W^{1, \infty}(0, T)}
\leq C(T) \norm{\Grad K_n}_{C^1_b}
\leq C(T) \norm{\Grad K}_{W^{1, \infty}}.
\end{align*}
Together with 
\begin{align*}
     \norm{\Grad K_n \ast \rho_{\Eps_k}(t, \cdot)}_{W^{1, \infty}(\Rd)}
\leq \norm{\Grad K_n}_{W^{1, \infty}}
\leq \norm{\Grad K}_{W^{1, \infty}},
\end{align*}
we have 
\begin{align*}
   \norm{\Grad K_n \ast \rho_{\Eps_k}}_{W^{1, \infty}([0, T) \times \R^d)}
\leq 
   \vpran{C(T) + 1} \norm{\Grad K}_{W^{1, \infty}}.
\end{align*}

Therefore, for any $t, s \in [0, T)$ and $x, y \in \R^d$,
\begin{align} \label{convg:K-n}
    \abs{\Grad K_n \ast \rho_{\Eps_k}(t, x) - \Grad K_n \ast \rho_{\Eps_k}(s, y)}
\leq 
   \vpran{C(T) + 1} \norm{\Grad K}_{W^{1, \infty}} \vpran{|t-s| + |x - y|}.
\end{align}
Since $\Grad K$ is continuous, we have that $\Grad K_n \to \Grad K$ uniformly on any compact set in $\R^d$. By
\begin{align*}
   \abs{\Grad K_n \ast \rho_{\Eps_k} (t, x)
           - \Grad K \ast \rho_{\Eps_k} (t, x)}
\leq 
   \sup_x \abs{\Grad K_n(x) - \Grad K(x)}, 
\end{align*}
we deduce that for any compact set $\widetilde \Omega \subset \R^d$
\begin{align*}
  \Grad K_n \ast \rho_{\Eps_k} (t, x) 
\stackrel{n\to\infty}{\longrightarrow}
  \Grad K \ast \rho_{\Eps_k} (t, x) \,,
\qquad \text{uniformly for $t \in [0, T), x \in \widetilde\Omega$, and $k \in \N$}.
\end{align*}
Hence, if we pass $n \to \infty$ in~\eqref{convg:K-n} over any compact set $\tilde\Omega \subset \R^d$, then
\begin{align} \label{convg:K-conv-rho}
    \abs{\Grad K \ast \rho_{\Eps_k}(t, x) - \Grad K \ast \rho_{\Eps_k}(s, y)}
\leq 
   \vpran{C(T) + 1} \norm{\Grad K}_{W^{1, \infty}} \vpran{|t-s| + |x - y|},
\end{align}
for any $t, s \in [0, T)$ and $x, y \in \widetilde\Omega$.

By Ascoli-Arzel\'a theorem, there exists a further subsequence (also denoted as $\rho_{\Eps_k}$) such that
\begin{align} \label{convg:5}
    \Grad K \ast \rho_{\Eps_k} \to \Grad K \ast \rho 
\qquad \text{ as } \Eps_k \to 0 \quad \text{ strongly in } L^\infty([0, T) \times \widetilde\Omega),
\end{align} 
for any compact set $\widetilde\Omega \subset \R^d$.

Now for every $t \in (0, T)$, we have
\begin{equation}
\begin{aligned}
&  \quad \,
\Biggl |\; \int_{\R^{d}} \Grad\psi \cdot
      \big(\vpran{\Grad K \ast \rho_{\Eps_k}} \rho_{\Eps_k}
                 - \vpran{\Grad K \ast \rho} \rho \big) \dx \Biggr |
\\
& \leq 
  \int_{\Omega_1(T)} \abs{\Grad\psi}
      \abs{\Grad K \ast \rho_{\Eps_k}
                 - \Grad K \ast \rho} \rho_{\Eps_k} \dx
 + \Biggl | \; \int_{\R^{d}} \Grad\psi \cdot \vpran{\Grad K \ast \rho}
                   \vpran{\rho_{\Eps_k} - \rho} \dx \Biggr | \nonumber
\\
& \leq  
   \norm{\Grad K \ast \rho_{\Eps_k}
                 - \Grad K \ast \rho}_{L^\infty(\Omega_1(T))} 
   \norm{\Grad\psi}_{L^\infty}
 + \Biggl | \; \int_{\Omega_1(T)} \Grad\psi \cdot \vpran{\Grad K \ast \rho}
                   \vpran{\rho_{\Eps_k} - \rho} \dx \Biggr | \,,
\end{aligned}
\end{equation}
where the first term converges to zero uniformly (in time) as $\Eps_k \to 0$ by~\eqref{convg:5} and the second term converges to zero pointwise in $t$ by~\eqref{convg:6}. This combined with \eqref{convg:4} and \eqref{eqn:rewrite} yields that for each $t \in (0, T)$, 
\begin{align*}
    \iint_{\R^d \times \R^d} \Grad \psi \cdot v f_{\Eps_k} \dx\dv
  \to 
  - \int_{\R^d}
        \Grad \psi \cdot \vpran{\Grad K \ast \rho} \rho \dx \qquad \text{ as } \Eps_k \to 0 \,.
\end{align*}

To conclude, let 
\begin{align*}
   \Omega_2(T) = \{v \in \R^d: (x, v) \in \Omega(T)\} \,.
\end{align*} 
Then $\Omega_2(T)$ is bounded for all $t \in (0, T)$ and we have the uniform (in $t$) bound 
\begin{align*}
    \Biggl | \; \; \iint_{\R^d \times \R^d} \Grad \psi \cdot v f_{\Eps_k} \dx\dv \Biggr |
\leq 
    C_2 \norm{\Grad\psi}_{L^\infty(\R^d)},
\end{align*}
where $C_2$ only depends on $\Omega_2(T)$. By the Lebesgue's dominated convergence theorem, we infer that 
\begin{align*}
   \int_0^T \iint_{\R^d \times \R^d} \Grad\psi \cdot v f_{\Eps_k} \dx\dv\dt
  \to - \int_0^T \iint_{\R^d \times \R^d} 
        \Grad\psi \cdot \vpran{\Grad K \ast \rho} \rho \dx\dt,
\qquad 
   \text{as $\Eps_k \to 0$}.
\end{align*}
Now \eqref{eq:rho-weak} follows from \eqref{eqn:rhoeps-weak} in the $\Eps_k \to 0$ limit. 

\smallskip
{\em Step 3.} Hence the limiting measure $\rho \in C([0, T); \CalP(\R^d))$ is a weak solution to~\eqref{eq:rho}. By \cite{AGS2005} (Lemma 8.1.6 in Chapter 8), $\rho$ is the push-forward of the initial density $\rho_0$ by the characteristic flow, i.e., 
$\rho = {\mathcal{T}}^{t}_{E_1[\rho]} \# \rho_0$ with the vector field given by $E_1[\rho] = -\Grad K \ast \rho \in L^\infty([0, T) \times \R^d)$. Moreover, since $\rho(t, \cdot)$ is compactly supported and narrowly continuous in time, we have that $\rho(t, \cdot) \in C([0, T); \CalP_1(\R^d))$, where the continuity is in the $W_1$ metric cf. Remark~\ref{rmk:W1}. 

We conclude that $\rho$ is the {\em unique} solution of \eqref{eq:rho} in the mass transportation sense \cite{CanizoCarrilloRosado2011}. Consequently, we infer that the {\em full} sequence $\rho_\Eps(t, \cdot)$ converges to $\rho(t, \cdot)$ with respect to the $W_1$ distance, for each $t \in [0, T)$, as desired.
\end{proof}


\subsection{Convergence of characteristic paths}
 Consider the solution $(x_\Eps(t),v_\Eps(t))$ of \eqref{eq:characteristics-eps}, that is, the characteristic paths defining the flow on $\Rd \times \Rd$ along which $\Epsf$ is being transported. We now investigate their limit as $\Eps \to 0$.

\begin{Theorem}[Convergence of characteristic paths]
\label{thm:conv-traj} Consider the measure-valued solution $\Epsf$ to \eqref{eq:kinetic} and a characteristic path $(x_\Eps(t),v_\Eps(t))$ that originates  from $(x_0,v_0) \in \operatorname{supp} f_0$ at $t=0$. Then,
\begin{equation}
\label{eqn:conv-trajx}
\lim_{\Eps \to 0} x_\Eps(t) = x(t), \quad \text{ for all } \quad 0\leq t \leq T, 
\end{equation}
where $x(t)$ is the characteristic trajectory of the limiting macroscopic equation \eqref{eq:rho} that starts at $x_0$, i.e., $x(t)$ satisfies
\[
\frac{dx}{dt} = -\nabla_x K \ast \rho, \quad x(0) = x_0.
\]
Also,
\begin{equation}
\label{eqn:conv-trajv}
 \lim_{\Eps \to 0}v_\Eps(t) = v(t), \quad \text{ for all } \quad 0< t \leq T,
\end{equation}
where 
\begin{equation}
\label{eqn:vlimit}
v(t) = -\nabla_xK \ast \rho(t,x(t)).
\end{equation}
\end{Theorem}
\begin{proof}
The task is to send $\Eps \to 0 $ in the characteristic system \eqref{eq:characteristics-eps}. Note that  \eqref{eq:characteristics-eps} does not fit immediately into the form \eqref{eqn:shorthand} needed for a direct application of Tikhonov's theorem (Theorem \ref{thm:Tikhonov}), as the right-hand-side of the $v$-equation depends on $\epsilon$ as well. 

Replace $\Epsrho$ by $\rho$ in the right-hand-side of \eqref{eq:characteristics-eps} to get the following system:
\begin{equation}
\begin{aligned} \label{eq:char-noeps}
     \frac{{\rm d} x}{\dt} & = v \,,
\\
     \Eps \frac{{\rm d} v}{\dt} & = - v - \Grad K \ast \rho.
\end{aligned}
\end{equation}

Denote by $(\tilde{x}_\Eps(t),\tilde{v}_\Eps(t))$ the solution of \eqref{eq:char-noeps} that starts from $(x_0,v_0)$. By Theorem \ref{thm:Tikhonov}, the convergence in \eqref{eqn:conv-trajx}-\eqref{eqn:vlimit}, which needs to be shown for ${x}_\Eps(t)$ and ${v}_\Eps(t)$, holds for $\tilde{x}_\Eps(t)$ and $\tilde{v}_\Eps(t)$. Hence, it would be enough to show that for a fixed $t>0$,
\begin{equation}
\label{conv:xv}
\lim_{\Eps \to 0} |x_\Eps(t)-\tilde{x}_\Eps(t)| = 0 \quad \text{ and } \quad \lim_{\Eps \to 0} |v_\Eps(t)-\tilde{v}_\Eps(t)| = 0.
\end{equation}

Indeed, from \eqref{eq:characteristics-eps} and \eqref{eq:char-noeps} we get
\[
\Eps \frac{{\rm d}}{\dt} (v_\Eps(t)-\tilde{v}_\Eps(t)) = -(v_\Eps(t)-\tilde{v}_\Eps(t)) - \nabla_x K \ast (\Epsrho - \rho).
\]
By integrating the above equation we find
\[
v_\Eps(t)-\tilde{v}_\Eps(t) = -\frac{1}{\Eps} \int_0^t e^{\frac{1}{\Eps}(s-t)} \, \nabla_x K \ast (\Epsrho - \rho) ds, \qquad \text{ for all } t \in [0,T].
\]
Using notations from the proof of Theorem \ref{thm:conv-sol} (see \eqref{def:Omega-1}), we have
\begin{align*}
|v_\Eps(t)-\tilde{v}_\Eps(t) &\leq \sup_{t\in[0,T]}\| \nabla_x K \ast (\Epsrho - \rho)\|_{L^\infty(\Omega_1(T))} \; \frac{1}{\Eps} \int_0^t e^{\frac{1}{\Eps}(s-t)} ds \\
& \leq \sup_{t\in[0,T]}\| \nabla_x K \ast (\Epsrho - \rho)\|_{L^\infty(\Omega_1(T))}, \qquad \text{ for all } t \in [0,T].
\end{align*}
By the uniform convergence shown in \eqref{convg:5} (which holds on the full sequence $\Epsrho$) we infer \eqref{conv:xv}, and hence, the desired result.
\end{proof}

The convergence of characteristic paths yields the limiting flow map $\mathcal{T}^t$ given by \\ $x_0 \stackrel{\mathcal{T}^t}{\longrightarrow} x(t)$. It is convenient in the calculations below to use the notation $x(t;x_0)$ to denote the limiting characteristic path $x(t)$ that starts at $x_0$. 

The next result characterizes the limiting densities.



\begin{Theorem}[Characterization of the limiting densities] 
\label{thm:rho-trans}
The limiting macroscopic density $\rho$ identified in Theorem \ref{thm:conv-sol} is the push-forward of the initial density $\rho_0$ by the limiting flow map $\mathcal{T}^t$, 
\begin{equation}
\label{eqn:rho-trans}
\rho = \mathcal{T}^t \# \rho_0.
\end{equation}

In addition,  for each $t\in [0,T)$, $\Epsf$ converges in the $W_1$ metric to a probability density $f(t,\cdot,\cdot)\in \mathcal{P}_1(\Rd \times \Rd)$:
\begin{equation}
\label{eqn:fconv}
   \Epsf \stackrel{W_1}{\longrightarrow} f
\qquad \text{ in } \mathcal{P}_1(\Rd \times \Rd) \quad \text{ as } \Eps \to 0.
\end{equation}
The limiting density $f$, with first marginal $\rho$, is given explicitly by: 
\begin{equation}
\label{eqn:f-delta}
f(t,x,v) = \rho(t,x) \delta (v+ \nabla K \ast \rho(t,x)).
\end{equation}
\end{Theorem}
\begin{proof}
The first part, expressed by equation \eqref{eqn:rho-trans}, follows from considerations made in Theorem \ref{thm:conv-sol} 
(see in particular Step 3 in the proof of Theorem \ref{thm:conv-sol}). However, we show below how it can be derived directly in the $\Eps \to 0$ limit of the kinetic equation.

The limiting behaviour of $\Epsf$ was not explicitly stated or needed in Theorem \ref{thm:conv-sol}, but follows by arguments similar to those used for $\Epsrho$ in Step 1 of the proof of Theorem \ref{thm:conv-sol}. Let us sketch this argument briefly.

Fix $\phi_1 \in C_c^1(0, T)$ and $\phi_2 \in C^1_b(\Rd \times \Rd)$, and let $\phi(t, x, v) = \phi_1(t)\phi_2(x,v)$ in \eqref{eq:f-Eps-weak}. Find
\begin{multline}
\label{eqn:feps-weak-1}
    \int_0^T \phi'_1(t)\iint_{\Rd \times \Rd} \phi_2(x,v) \Epsf(t, x,v) \dx \dv \dt
    =\\
  - \int_0^T \phi_1(t) \iint_{\Rd \times \Rd}  \left[ \nabla_x \phi_2 \cdot v  - \frac{1}{\Eps} \nabla_v \phi_2 \cdot (v + \Grad K \ast \Epsrho) \right] \Epsf \dx\dv\dt.
\end{multline}
Denoting
\begin{equation*}
\tilde{\eta}_1(t) = \iint_{\Rd \times \Rd} \phi_2(x,v) \Epsf(t, x,v) \dx \dv,
\end{equation*}
then, by \eqref{eqn:feps-weak-1}, the weak derivative of $\tilde{\eta}_1$ is given by
\begin{equation*}    
    \tilde{\eta}'_1(t) = \iint_{\R^d \times \R^d} \left[ \nabla_x \phi_2 \cdot v  - \frac{1}{\Eps} \nabla_v \phi_2 \cdot (v + \Grad K \ast \Epsrho) \right]   \Epsf \dx\dv 
    \in L^\infty(0, T).
\end{equation*}
In the above, the boundness of the term that contains $\Eps$ follows from Proposition \ref{thm:bound-f-eps}.

Since $\tilde{\eta}_1(t)$ is uniformly bounded in $W^{1, \infty}(0, T)$, it converges uniformly on a subsequence. Hence, given any $\phi_2 \in C^1_b(\Rd \times \Rd)$, there exists a subsequence $\Eps_k$ and $\tilde{\eta}_2(t) \in C([0, T))$ such that 
\begin{equation}
\label{eqn:convfk}
    \iint_{\R^d \times \Rd} \phi_2(x,v) f_{\Eps_k}(t, x,v) \dx \dv 
\to \tilde{\eta}_2(t) \quad \text{uniformly in } C([0, T)) \quad \text{ as } \Eps_k \to 0.
\end{equation}

Similar to arguments used for $\Epsrho$, we note that the sequence  $\Epsf(t,\cdot,\cdot) \in \mathcal{P}_1(\Rd \times \Rd) $  is tight, 
and hence, for each $t \in [0, T)$, $\Epsf(t,\cdot,\cdot)$ converges  in the $W_1$ metric on a subsequence 
to a probability measure $f(t, \cdot,\cdot) \in \CalP_1(\R^d \times \Rd)$. 

By \eqref{eqn:convfk}, one can then argue similarly as in the proof of Theorem \ref{thm:conv-sol} that 
\begin{equation*}
\label{eqn:convfku}
    \iint_{\Rd \times \Rd} \phi_2(x,v) f_{\Eps_k}(t, x,v) \dx \dv 
\to \iint_{\Rd \times \Rd } \phi_2(x,v) f(t, x,v) \dx \dv \quad \text{uniformly on } [0, T) \quad \text {as } \Eps_k \to 0,
\end{equation*}
for any $\phi_2 \in C^1_b(\Rd \times \Rd)$ where $\Eps_k$ is independent of $t$. Further, by the uniqueness of $\tilde{\eta}_2(t)$, we have that at each $t \in [0, T)$ 
\begin{align}
 \label{convf:6}
     f_{\Eps_k} (t, \cdot,\cdot)  \stackrel{W_1}{\longrightarrow} f(t, \cdot, \cdot)
\qquad \text{ in } \CalP_1(\Rd \times \Rd) \quad \text{ as } \Eps_k \to 0.
\end{align}
This proves the convergence \eqref{eqn:fconv} on a subsequence. To upgrade it to convergence on the full sequence $\Eps\to 0$ we use the uniqueness of $f$, as derived from the arguments below.

Since $\Epsf(t) = \mathcal{T}^{t,\Eps}_{E[\Epsf]} \# f_0$, by \eqref{eqn:mass-transport} it holds that
\begin{equation}
\label{eqn:trans2}
\iint_{\Rd \times \Rd} \zeta (x,v) f_{\Eps_k}(t,x,v) dx dv = \iint_{\Rd \times \Rd} \zeta(\mathcal{T}^{t,\Eps_k}_{E[f_{\Eps_k}]}(X,V)) f_0(X,V) dX dV,
\end{equation}
for all $\zeta \in C_b(\Rd \times \Rd)$.

By the weak convergence of $f_{\Eps_k}$, the left-hand-side of \eqref{eqn:trans2} converges as $\Eps_k \to 0$:
\[
\iint_{\Rd \times \Rd} \zeta (x,v) f_{\Eps_k}(t,x,v) dx dv \to \iint_{\Rd \times \Rd} \zeta (x,v) f(t,x,v) dx dv.
\]
Due to convergence of trajectories \eqref{eqn:conv-trajx}-\eqref{eqn:vlimit}, the right-hand-side of \eqref{eqn:trans2} converges by Lebesgue's dominated convergence theorem,
\[
\iint_{\Rd \times \Rd} \zeta(\mathcal{T}^{t,\Eps_k}_{E[f_{\Eps_k}]}(X,V)) f_0(X,V) dX dV \to \iint_{\Rd \times \Rd} \zeta(x(t;X),-\nabla K \ast \rho (t,x(t;X))) f_0(X,V) dX dV,
\]
as $\Eps_k \to 0$. Combining the two, we find
\begin{equation}
\label{eqn:equal1}
 \iint_{\Rd \times \Rd} \zeta (x,v) f(t,x,v) dx dv = \iint_{\Rd \times \Rd} \zeta(x(t;X),-\nabla K \ast \rho (t,x(t;X))) \rho_0(X,V) dX,
\end{equation}
for all $\zeta \in C_b(\Rd \times \Rd)$.

First note that \eqref{eqn:rho-trans} can be derived from \eqref{eqn:equal1}. Indeed, choose $\zeta(x,v) = \varphi(x)$ in \eqref{eqn:equal1} to find 
\[
\int_{\Rd} \varphi(x) \rho(t,x) dx = \int_{\Rd} \varphi (x(t;X)) \rho_0(X) dX,
\]
for all $\varphi \in C_b(\Rd)$. The equation above represents exactly the mass transport given by \eqref{eqn:rho-trans}.

Now, observe that \eqref{eqn:f-delta} is equivalent to 
\begin{equation*}
 \iint_{\Rd \times \Rd} f(t,x,v) \zeta(x,v) dx dv = \int_{\Rd}  \zeta(x,-\nabla K \ast \rho(t,x)) \rho(t,x) dx,
\end{equation*}
for all test functions $\zeta \in C_b(\Rd \times \Rd)$, which can be inferred immediately from \eqref{eqn:rho-trans} and \eqref{eqn:equal1}.

Finally, the {\em unique} explicit representation of the limiting density $f$ implies that the convergence in \eqref{convf:6} holds on the full sequence $\Epsf$, as desired in \eqref{eqn:fconv}.
\end{proof}


\bibliographystyle{plain}
\bibliography{lit}

\def\cprime{$'$}
\begin{thebibliography}{10}

\bibitem{Amb2003}
L.~Ambrosio.
\newblock Lecture notes on optimal transport problem.
\newblock In J.~Rodrigues P.~Colli, editor, {\em Mathematical aspects of
  evolving interfaces}, volume 1812 of {\em CIME summer school in Madeira
  (Pt)}, pages 1--52. Springer, 2003.

\bibitem{Balague_etalARMA}
D.~Balagu{\'e}, J.~A. Carrillo, T.~Laurent, and G.~Raoul.
\newblock Dimensionality of local minimizers of the interaction energy.
\newblock {\em Arch. Ration. Mech. Anal.}, 209(3):1055--1088, 2013.

\bibitem{BaCaLaRa2013}
D.~Balagu{\'e}, J.~A. Carrillo, T.~Laurent, and G.~Raoul.
\newblock Nonlocal interactions by repulsive-attractive potentials: radial
  ins/stability.
\newblock {\em Phys. D}, 260:5--25, 2013.

\bibitem{BertozziCarilloLaurent}
Andrea~L. Bertozzi, Jos{\'e}~A. Carrillo, and Thomas Laurent.
\newblock Blow-up in multidimensional aggregation equations with mildly
  singular interaction kernels.
\newblock {\em Nonlinearity}, 22(3):683--710, 2009.

\bibitem{BertozziLaurent}
Andrea~L. Bertozzi and Thomas Laurent.
\newblock Finite-time blow-up of solutions of an aggregation equation in
  {$\mathbf{R}^n$}.
\newblock {\em Comm. Math. Phys.}, 274(3):717--735, 2007.

\bibitem{BeLaRo2011}
Andrea~L. Bertozzi, Thomas Laurent, and Jes{\'u}s Rosado.
\newblock {$L^p$} theory for the multidimensional aggregation equation.
\newblock {\em Comm. Pure Appl. Math.}, 64(1):45--83, 2011.

\bibitem{Billingsley}
P.~Billingsley.
\newblock {\em Weak convergence of measures: {A}pplications in probability}.
\newblock Society for Industrial and Applied Mathematics, Philadelphia, Pa.,
  1971.

\bibitem{BodnarVelasquez1}
M.~Bodnar and J.~J.~L. Velasquez.
\newblock Derivation of macroscopic equations for individual cell-based models:
  a formal approach.
\newblock {\em Math. Meth. Appl. Sci. (M3AS)}, 28:1757--1779, 2005.

\bibitem{BodnarVelasquez2}
M.~Bodnar and J.~J.~L. Velazquez.
\newblock An integro-differential equation arising as a limit of individual
  cell-based models.
\newblock {\em J. Differential Equations}, 222(2):341--380, 2006.

\bibitem{Burger:DiFrancesco}
Martin Burger and Marco Di~Francesco.
\newblock Large time behavior of nonlocal aggregation models with nonlinear
  diffusion.
\newblock {\em Netw. Heterog. Media}, 3(4):749--785, 2008.

\bibitem{Camazine_etal}
Scott Camazine, Jean-Louis Deneubourg, Nigel~R. Franks, James Sneyd, Guy
  Theraulaz, and Eric Bonabeau.
\newblock {\em Self-organization in biological systems}.
\newblock Princeton Studies in Complexity. Princeton University Press,
  Princeton, NJ, 2003.
\newblock Reprint of the 2001 original.

\bibitem{CanizoCarrilloRosado2011}
J.~A. Ca{\~n}izo, J.~A. Carrillo, and J.~Rosado.
\newblock A well-posedness theory in measures for some kinetic models of
  collective motion.
\newblock {\em Math. Models Methods Appl. Sci.}, 21(3):515--539, 2011.

\bibitem{Figalli_etal2011}
J.~A. Carrillo, M.~DiFrancesco, A.~Figalli, T.~Laurent, and D.~Slep{\v{c}}ev.
\newblock Global-in-time weak measure solutions and finite-time aggregation for
  nonlocal interaction equations.
\newblock {\em Duke Math. J.}, 156(2):229--271, 2011.

\bibitem{CarrilloVecil2010}
Jos{\'e}~A. Carrillo, Massimo Fornasier, Giuseppe Toscani, and Francesco Vecil.
\newblock Particle, kinetic, and hydrodynamic models of swarming.
\newblock In {\em Mathematical modeling of collective behavior in
  socio-economic and life sciences}, Model. Simul. Sci. Eng. Technol., pages
  297--336. Birkh\"auser Boston, Inc., Boston, MA, 2010.

\bibitem{ChFeTo2014}
R.~Choksi, R.~C. Fetecau, and I.~Topaloglu.
\newblock On minimizers of interaction functionals with competing attractive
  and repulsive potentials.
\newblock {\em Ann. I. H. Poincar\'{e} - AN}, 2014.
\newblock DOI: 10.1016/j.anihpc.2014.09.004 (in press).

\bibitem{EvFeRy2014}
J.~Evers, R.~C. Fetecau, and L.~Ryzhik.
\newblock Anisotropic interactions in a first-order aggregation model: a proof
  of concept,.
\newblock 2014.
\newblock preprint.

\bibitem{FeRa10}
Klemens Fellner and Ga{\"e}l Raoul.
\newblock Stable stationary states of non-local interaction equations.
\newblock {\em Math. Models Methods Appl. Sci.}, 20(12):2267--2291, 2010.

\bibitem{FeHu13}
R.~C. Fetecau and Y.~Huang.
\newblock Equilibria of biological aggregations with nonlocal
  repulsive-attractive interactions.
\newblock {\em Phys. D}, 260:49--64, 2013.

\bibitem{FeHuKo11}
R.~C. Fetecau, Y.~Huang, and T.~Kolokolnikov.
\newblock Swarm dynamics and equilibria for a nonlocal aggregation model.
\newblock {\em Nonlinearity}, 24(10):2681--2716, 2011.

\bibitem{G}
R.~Glassey.
\newblock {\em The {C}auchy {P}roblem in {K}inetic {T}heory}.
\newblock SIAM, Philadelphia, PA, 1996.

\bibitem{Haile1992}
J.M. Haile.
\newblock {\em Molecular Dynamics Simulation: Elementary Methods}.
\newblock John Wiley and Sons, Inc., New York, 1992.

\bibitem{Helbing2001}
D.~Helbing.
\newblock Traffic and related self-driven many particle systems.
\newblock {\em Rev. Mod. Phys.}, 73:1067--1141, 2001.

\bibitem{HoPu2005}
Darryl~D. Holm and Vakhtang Putkaradze.
\newblock Aggregation of finite-size particles with variable mobility.
\newblock {\em Phys Rev Lett.}, 95:226106, 2005.

\bibitem{Jackson2010}
M.~O. Jackson.
\newblock {\em Social and {E}conomic {N}etworks}.
\newblock Princeton University Press, 2010.

\bibitem{JiEgerstedt2007}
M.~Ji and M.~Egerstedt.
\newblock Distributed coordination control of multi-agent systems while
  preserving connectedness.
\newblock {\em IEEE Trans. Robot.}, 23(4):693--703, 2007.

\bibitem{KoSuUmBe2011}
Theodore Kolokolnikov, Hui Sun, David Uminsky, and Andrea~L. Bertozzi.
\newblock A theory of complex patterns arising from 2{D} particle interactions.
\newblock {\em Phys. Rev. E, Rapid Communications}, 84:015203(R), 2011.

\bibitem{AGS2005}
G.~Savar\'{e} L.~Ambrosio, N.~Gigli.
\newblock {\em Gradient Flows in Metric Spaces and in the Space of Probability
  Measures}.
\newblock Birkh\"auserauser, 2005.

\bibitem{Laurent2007}
Thomas Laurent.
\newblock Local and global existence for an aggregation equation.
\newblock {\em Comm. Partial Differential Equations}, 32(10-12):1941--1964,
  2007.

\bibitem{LeToBe2009}
Andrew~J. Leverentz, Chad~M. Topaz, and Andrew~J. Bernoff.
\newblock Asymptotic dynamics of attractive-repulsive swarms.
\newblock {\em SIAM J. Appl. Dyn. Syst.}, 8(3):880--908, 2009.

\bibitem{M&K}
A.~Mogilner and L.~Edelstein-Keshet.
\newblock A non-local model for a swarm.
\newblock {\em J. Math. Biol.}, 38:534--570, 1999.

\bibitem{MotschTadmor2014}
Sebastien Motsch and Eitan Tadmor.
\newblock Heterophilious dynamics enhances consensus,.
\newblock 2014.
\newblock preprint.

\bibitem{SR2014}
L.~Saint-Raymond.
\newblock A mathematical {PDE} perspective on the {C}hapman-{E}nskog expansion.
\newblock {\em Bull. Amer. Math. Soc.}, 51:247--275, 2014.

\bibitem{Tikhonov1952}
A.~N. Tikhonov.
\newblock Systems of differential equations containing small parameters in the
  derivatives.
\newblock {\em Mat. Sb. (N.S.)}, 31(73):575--586, 1952.

\bibitem{TBL}
C.~M. Topaz, A.~L. Bertozzi, and M.~A. Lewis.
\newblock A nonlocal continuum model for biological aggregation.
\newblock {\em Bull. Math. Bio.}, 68:1601--1623, 2006.

\bibitem{Toscani2000}
Giuseppe Toscani.
\newblock One-dimensional kinetic models of granular flows.
\newblock {\em M2AN Math. Model. Numer. Anal.}, 34(6):1277--1291, 2000.

\bibitem{Vasileva1963}
A.~B. Vasil{\cprime}eva.
\newblock Asymptotic behaviour of solutions of certain problems for ordinary
  non-linear differential equations with a small parameter multiplying the
  highest derivatives.
\newblock {\em Uspehi Mat. Nauk}, 18(3 (111)):15--86, 1963.

\bibitem{Brecht_etal2011}
James von Brecht, David Uminsky, Theodore Kolokolnikov, and Andrea Bertozzi.
\newblock Predicting pattern formation in particle interactions.
\newblock {\em Math. Models Methods Appl. Sci.}, 22(Supp. 1):1140002, 2012.

\bibitem{BrechtUminsky2012}
James~H. von Brecht and David Uminsky.
\newblock On soccer balls and linearized inverse statistical mechanics.
\newblock {\em J. Nonlinear Sci.}, 22(6):935--959, 2012.

\end{thebibliography}


\end{document}